\documentclass{amsart}

\usepackage{amssymb}
\usepackage{amsmath}
\usepackage{amsthm}				
\usepackage{tikz}					
\usepackage{booktabs}
\usepackage{makecell}

\usepackage{xstring,xifthen}
\usepackage{ifthen}

\usepackage{verbatim} 			% comments
\newtheorem{thm}{Theorem}
\newtheorem{prop}{Proposition}[section]
\newtheorem{lem}[prop]{Lemma}
\newtheorem{cor}[prop]{Corollary}

\newtheorem{conj}[prop]{Conjecture}
\newtheorem*{conj*}{Conjecture}

\newcommand{\rank}{\mathrm{rank}}
\newcommand{\PGL}[1]{PGL\left(2,#1\right)} 		
\newcommand{\PSL}[1]{PSL\left(2,#1\right)}		
	
\newcommand{\PSiL}[1]{P\Sigma L\left(2,#1\right)}	
\newcommand{\GL}[1]{GL\left(2,#1\right)}		
\newcommand{\SL}[1]{SL\left(2,#1\right)}		
\newcommand{\Z}[1]{Z\left(#1\right)}			

\newcommand{\GF}[1]{\mathbb{F}_{#1}}			
\newcommand{\gen}[1]{\left\langle#1\right\rangle}	
\newcommand{\Gcd}[1]{\left(#1\right)}
\newcommand{\Lcm}[1]{\mathrm{lcm}\left(#1\right)}
\newcommand{\Max}[1]{\mathrm{max}\left(#1\right)}
\newcommand{\Tr}[1]{\mathrm{Tr}\left(#1\right)}

\newcommand{\C}[1]{C_{#1}}	% Cyclic group

\newcommand{\order}[1]{\left|#1\right|}

\newcommand{\rk}{n}
\newcommand{\PG}[1]{PG\left(1,#1\right)}
\newcommand{\eul}[1]{\phi\left(#1\right)}

\newcommand{\M}[1]{M\left(#1\right)}

\newcommand{\Det}[1]{\mathrm{det}\left(#1\right)}
\newcommand{\D}[1]{D_{#1}}
\newcommand{\prim}{j}
\newcommand{\s}[1]{
\ifthenelse{#1 = 1}
{\alpha}
{\ifthenelse{#1 = 2}
  {\beta}
  {\ifthenelse{#1 = 3}
    {\gamma}
    {\delta}
    }
  }
}

\newcommand{\set}[1]{\left\{#1\right\}}
\newcommand{\Aut}[1]{
\StrLen{#1}[\mystring]
\ifthenelse{\mystring > 1}
{\mathrm{Aut}\left(#1\right)}
{\mathrm{Aut}#1}
}

\newcommand{\AGL}[1]{\mathrm{AGL}\left(1,#1\right)}

\newcommand{\Sym}[1]{S_#1}

\newcommand{\poly}{\mathcal{P}}

\newcommand{\Autgp}{\Gamma}

\newcommand{\GFb}[1]{\bar{\mathbb{F}}_{#1}}

\newcommand{\scal}[1]{
\ifthenelse{#1 = 1}
{\lambda}
{\ifthenelse{#1 = 2}
  {\mu}
  {\ifthenelse{#1 = 3}
    {\nu}
    {\kappa}
    }
  }
}

\newcommand{\ab}[1]{
\StrLen{#1}[\mystring]
\ifthenelse{\mystring > 1}
{\left(#1\right)^{ab}}
{#1^{ab}}
}

\def\subsection{\@startsection{subsection}{3}%
  \z@{.5\linespacing\@plus.7\linespacing}{-.5em}%
    {\normalfont\itshape}}
\makeatletter
\def\subsection{\@startsection{subsection}{2}%
\z@{1.5\linespacing\@plus.7\linespacing}{1.5\linespacing}%
  {\normalfont\bfseries}}

\def\subsubsection{\@startsection{subsubsection}{2}%
\z@{1.5\linespacing\@plus.7\linespacing}{1.5\linespacing}%
  {\normalfont\bfseries}}

\def\section{\@startsection{section}{1}%
  \z@{1.5\linespacing\@plus\linespacing}{1.5\linespacing}%
  {\normalfont\scshape\centering}}

\makeatother

\title[Projective linear groups and chiral polytopes]{Projective linear groups as automorphism groups of chiral polytopes}
\author{Dimitri Leemans}
\address{Dimitri Leemans, University of Auckland,
Department of Mathematics,
Private Bag 92019,
Auckland, New Zealand}
\email{d.leemans@auckland.ac.nz}

\author{J\'er\'emie Moerenhout}
\address{J\'er\'emie Moerenhout, University of Auckland,
Department of Mathematics,
Private Bag 92019,
Auckland, New Zealand}
\email{j.moerenhout@auckland.ac.nz}

\author{Eugenia O'Reilly-Regueiro}
\address{Eugenia O'Reilly-Regueiro, Instituto de Matematicas, Universidad Nacional Autonoma de Mexico,
Circuito Exterior, C.U. Coyoacan 04510, Mexico D.F.}
\email{eugenia@im.unam.mx}


\begin{document}

%\keywords{}
%\subjclass[]{}

% ***************************************************************************************************

\begin{abstract}
It is already known that the automorphism group of a chiral polyhedron is never isomorphic to $\PSL{q}$ or $\PGL{q}$ for any prime power $q$.
In this paper, we show that $\PSL{q}$ and $\PGL{q}$ are never automorphism groups of chiral polytopes of rank at least $5$.
Moreover, we show that $\PGL{q}$ is the automorphism group of at least one chiral polytope of rank $4$ for every $q\geq5$.
Finally, we determine for which values of $q$ the group $\PSL{q}$ is the automorphism group of a chiral polytope of rank $4$, except when $q=p^d\equiv3\pmod{4}$ where $d>1$ is not a prime power, in which case the problem remains unsolved.

%%A similar statement is not known in higher rank.
%For each prime power $q$, we determine whether there exists a chiral $4$-polytope whose automorphism group is isomorphic to $\PSL{q}$ or $\PGL{q}$.
%%The purpose of this paper is to show that neither $\PSL{q}$ or $\PGL{q}$ is the automorphism group of a chiral polytope of rank higher than four.
%We show that four is the maximal rank for both $\PSL{q}$ and $\PGL{q}$.
%We investigate the rank four and show that if $q>4$ then $\PGL{q}$ is always the automorphism group of a chiral $4$-polytope.
%In contrast, there exist infinitely many $q$ for which $\PSL{q}$ is not the automorphism group of a chiral $4$-polytope.
%We determine conditions on $q$ for which $\PSL{q}$ is a chiral $4$-polytope.
\end{abstract}%

\maketitle

%\begin{commentJ}
%Notations:
%\begin{multicols}{2}
%\begin{itemize}
%\item $\poly$ : polytope
%\item $\Phi$ : base flag
%\item $n$ : rank
%\item $\sigma_i$ : generator of a chiral polytope
%\item $\tau$: other element of $\PSL{q}$
%\item $\mathcal{F}_i$ : $i$-face
%\item $p_i$ : schlafli symbol
%\item $I$, $J$ : subset of $\mathbb{Z}$
%\item $\Gamma$ : Automorphism group of a polytope
%\item $\rho$ : involutory automorphism that inverts ...
%\item $G$ : Group
%\item $p$ : prime
%\item $q$ : prime power
%\item $d$ : power of $p$ ($q=p^d$)
%\item $k$ : divisor of $q\pm1$
%\item $\E{q}$, $C_k$, $D_k$, $A_k$, $S_k$ : usual groups
%\item $e$, $f$, $e_i$ : divisor of $d$
%\item $H$, $K$ : subgroups of $G$
%\item $\tilde{e}$ : less common divisor
%\item $\PG{q}$ : projective line
%\item $\order{\sigma}$ : order
%\item $j$ : primitive element of a finite field
%\item $\GF{q}$, $F$ : finite field
%\item $\alpha,\beta$ : scalars
%\item $\xi$ : unknown scalar
%\item $[.,.,.]$ : schlafli type
%\item $[.,.,.]^+$ : even Coxeter group
%\item $\phi$ : Euler totient function
%\item $\Sigma$ : representative of $\sigma$ in $\SL{q}$.
%\item $\omega_i$ : trace
%\item $\psi$ : element of $\mathrm{Hom}(\Gamma_0^+,\overline{L})$
%\item $\theta$ : element of $\mathrm{Hom}(\Gamma^+,\overline{L})$
%
%
%
%\end{itemize}
%\end{multicols}
%\end{commentJ}
\section{Introduction}
Polyhedra, and their generalisation to higher ranks, polytopes, are geometric objects that have been studied since the  time of the Greeks.
Abstract polytopes ge\-ne\-ra\-li\-se the concept of polytopes, defining them as partially ordered sets that satisfy some axioms (see Section~\ref{arp} for more details).
The study of highly symmetric abstract polytopes constructed from families of almost simple groups has had a lot of attention in the last decade. 
In the case where the polytopes have the maximum level of symmetries, also called regular polytopes, results have been obtained for several families including transitive groups~\cite{transitive}, the symmetric and alternating groups~\cite{FL2011,FLM2012b,FLM2012, extension, anproof}, the Suzuki groups~\cite{KL2009,Lee2005a}, the small Ree groups~\cite{LSV2015}, the almost simple groups with socle $\PSL{q}$~\cite{DiJuTho,LS2005a, LS2008a}, groups of type $PSL(3,q)$ and $PGL(3,q)$~\cite{BV2010}, $PSL(4,q)$~\cite{BL2015} and several sporadic simple groups~\cite{CL2013,CLM2012,HH2010,LM2011,LV2006}.

In the case where the polytopes have maximum level of rotational symmetries but no reflections, also called the chiral case, much less is known.
Hartley, Hubard and Leemans produced in~\cite{HHL2012} an atlas of chiral polytopes for small almost simple groups.
Hubard and Leemans studied the family of almost simple groups of Suzuki type in~\cite{HL2014}. Conder, Hubard, O'Reilly-Regueiro and Pellicer studied the alternating and symmetric groups~\cite{CHOP2015,CHOP2}.
Chiral polytopes of rank three (also called polyhedra) and almost simple groups with socle $\PSL{q}$ have been shown not to exist for the groups $\PSL{q}$ and $\PGL{q}$~\cite{Hypermaps, ELO2015}.
More recently, the first two authors of this article showed in~\cite{LM2016} that if $q\geq 4$ and $G$ is an almost simple group with socle $\PSL{q}$,  then $G$ is the automorphism group of (at least) one chiral polyhedron if and only if $G$ is not one of $\PSL{q}$, $\PGL{q}$, $\M{1,9}$ or $\PSiL{4}\cong \PGL{5}$.

The aim of this article is to determine what is the maximal rank of a chiral polytope with automorphism group $\PSL{q}$ or $\PGL{q}$.
Our results are summarised in the following theorem.

\begin{thm}\label{maintheo}
Let $q\geq 4$ be a prime power.
\begin{enumerate}
\item
Neither $\PSL{q}$ nor $\PGL{q}$ are automorphism groups of chiral polytopes of rank at least 5.\label{main:rg5}
\item
The group $\PGL{q}$ is the automorphism group of at least one chiral polytope of rank 4 if and only if $q\geq5$. \label{main:PGL}
\item
If $q\equiv1\pmod{4}$, then $\PSL{q}$ is the automorphism group of at least one chiral polytope of rank 4 if and only if $q\geq13$. \label{PSL_1mod4}
\item \label{PSL_3mod4}
Let $d$ be an (odd) prime power or $d=1$ and $q=p^d\equiv3\pmod{4}$ be a prime power.
Then $\PSL{q}$ is the automorphism group of at least one chiral polytope of rank 4 if and only if at least one of the following conditions on $q$ is satisfied:
\begin{enumerate}
\item $q=p\equiv11,19\pmod{20}$ and $p\equiv1,3,4,5,9\pmod{11}$ and $3\pm2\sqrt{5}$ are both squares in $\GF{p}$;\label{1119mod20_1}
\item $q=p\equiv11,19\pmod{20}$ and $p\equiv2,6,7,8\pmod{11}$;\label{1119mod20_2}
\item $q=p\equiv11,19\pmod{20}$ and $p\equiv1,4,5,6,7,9,11,16,17\pmod{19}$ and $\frac{7\pm5\sqrt{5}}{2}$ are both squares in $\GF{p}$;\label{1119mod20_3}
\item $q=p\equiv11,19\pmod{20}$ and $p\equiv2,3,8,10,12,13,14,15,18\pmod{19}$;\label{1119mod20_4}
\item $q=p\equiv31,39\pmod{40}$.\label{3139mod40}
\end{enumerate}
\end{enumerate}
\end{thm}

Note that, in case \eqref{PSL_3mod4}, ``odd" is in parentheses since $q\equiv3\pmod{4}$ is true only if $d$ is odd.
Also, if $q$ is even, then $\PSL{q}\cong\PGL{q}$ and one may refer to case \eqref{main:PGL}.
Before continuing, observe that each condition from \eqref{1119mod20_1} to \eqref{3139mod40} is necessary (that is, not redundant).
Indeed, for each case, there exists an integer which is the smallest prime power such that the condition is necessary.
Those integers are as follows:
%If $q=p^d$, with $d>1$ and $q\equiv3\pmod{4}$, then $\PSL{q}$ is not the automorphism group of a chiral $4$-polytope.
\begin{align*}
\eqref{1119mod20_1}:q=619;&&\eqref{1119mod20_2}:q=139;&&\eqref{1119mod20_3}:q=131;&&\eqref{1119mod20_4}:q=179;&&\eqref{3139mod40}:q=631.
\end{align*}
%the smallest prime for which condition \eqref{1119mod20_2} is necessary is $q=139$;
%the smallest prime for which condition \eqref{1119mod20_3} is necessary is $q=131$;
%the smallest prime for which condition \eqref{1119mod20_4} is necessary is $q=179$ and the smallest prime for which condition \eqref{3139mod40} is necessary is $q=631$.

In order to make the previous result complete, it remains to determine whether $\PSL{q}$ is the automorphism group of a chiral polytope of rank $4$, when $q=p^d\equiv3\pmod{4}$ and $d>1$ is not a prime power, that is when $d$ is the product of at least two odd prime powers of distinct characteristic.
We provide here a conjecture, and strong evidence for it in Section \ref{sec:conj}.

\begin{conj*}
Let $G\cong \PSL{q}$ with $d>1$.
Then $G$ is the automorphism group of a chiral polytope of rank 4 if and only if one of the following conditions on $q$ is satisfied:
\begin{enumerate}
\item $4<q\equiv0\pmod{2}$; \label{0mod2_}
\item $9<q\equiv1\pmod{4}$; \label{1mod4_}
\item $q=p^d\equiv3\pmod{4}$ where $d>1$ is not a prime power. \label{pd3mod4}
\end{enumerate}
\end{conj*}
The cases (1) and (2) have been shown to be true in this paper.
The smallest value of $q$ for which case (3) applies is for $q=3^{15}$.
This does not make possible any computation to test this conjecture.
However, we construct representative matrices in Section \ref{sec:conj}, that we strongly believe generate the automorphism group of a chiral polytope.

%\begin{commentJ}
%What I want to say here is that each condition is necessary in the sense that if you remove it, then the theorem is false.
%I wanted to precise this to emphasize that none of the above conditions (a) to (g) is "useless".
%For example I could have added a condition (h) $q=p\equiv1,9\pmod{40}$ and the Theorem is still true.
%But (h) is a contained in (b) and so (h) is not "necessary".
%I will think of another way of saying that, since Dimitri also told me that we was confused with those "necessary"
%\end{commentJ}

The paper is organised as follows.
In Section~\ref{basic}, we recall the basic notions about abstract polytopes and projective linear groups needed to understand this article.
In Section~\ref{rank5}, we prove part \eqref{main:rg5} of Theorem~\ref{maintheo}.
In Section~\ref{sec:PGL4}, we prove part \eqref{main:PGL} of Theorem~\ref{maintheo} by constructing examples for each $q>4$.
In Section~\ref{sec:PSL4}, we prove parts \eqref{PSL_1mod4} and \eqref{PSL_3mod4} of Theorem~\ref{maintheo}.
We conclude the paper, in Section~\ref{sec:conj}, with some open problems and ideas to solve them.
\section{Basic notions}\label{basic}

\subsection{Abstract polytopes}\label{arp}
We refer to \cite{McMullen} for further information about the theory of abstract polytopes and abstract regular polytopes.

An abstract polytope $(\mathcal{P}, \leq)$ of rank $\rk$ is a partially ordered set whose elements are called {\em faces}, with a strictly monotone rank function $\rank:\mathcal{P}\rightarrow\{-1,0, \ldots, \rk\}$ and satisfying four axioms.

\begin{enumerate}

\item[(P1)] $\mathcal{P}$ has a unique face $F_{-1}$ corresponding to the empty set and a unique face $F_\rk$ corresponding to $\mathcal{P}$;\label{p1}
\item[(P2)] every maximal chain of $\mathcal{P}$ has $\rk+2$ faces, one of each rank;\label{p2}
\end{enumerate}

 A {\em flag} is a maximal chain of $\mathcal{P}$. Two flags are {\em adjacent} if they differ in just one
face, they are $i$-adjacent if the rank of the different face is $i$.
\begin{enumerate}
\item[(P3)] $\mathcal {P}$ is {\em strongly flag connected}, that is, for any two flags $\Phi$, $\Psi$ of $\mathcal{P}$, there exists a sequence of adjacent flags going from $\Phi$ to $\Psi$;\label{p3}

\item[(P4)] $\mathcal{P}$ satisfies the {\em diamond condition}, that is, given any two faces $J < K$ with $\rank(J) = \rank(K) - 2$, there are exactly two faces $I$, $I'$ of rank $\rank(K)-1$ such that $J < I,I' < K$.\label{p4}
\end{enumerate}

An {\em automorphism} of $\poly$ is a permutation of the flags that preserves the order and the rank function.
The set of automorphisms of $\poly$ is a group called the {\em automorphism group} of $\poly$ and denoted by $\Gamma(\poly)$.
By properties (P3) and (P4), any automorphism of P is determined by its action
on a given flag.

If $\Gamma(\mathcal{P})$ has a unique orbit on the set of flags of $\mathcal{P}$, we say that $\mathcal{P}$ is {\em regular}.
In that case, fixing a {\em base flag} $\Phi$ yields a set $\{\rho_0, \ldots, \rho_{\rk-1}\}$ of involutions which generate $\Gamma(\mathcal{P})$, where for each $i=0, \ldots, \rk-1$, the involution $\rho_i$ maps $\Phi$ to its unique $i$-adjacent flag, that is the unique flag distinct from $\Phi$ whose $j$-elements with $j\neq i$ are the same as those of $\Phi$.
The {\em rotation subgroup} (or {\em even subgroup}) $\Gamma^+(\mathcal{P})$ of $\Gamma(\mathcal{P})$ of words of even length of $\rho_0, \ldots, \rho_{\rk-1}$ has index at most 2 in $\Gamma(\mathcal{P})$, and is generated by $\{\sigma_1, \ldots,  \sigma_{\rk-1}\}$, where $\sigma_i:=\rho_{i-1}\rho_i$. If $\Gamma^+(\mathcal{P})$ has index 2 in $\Gamma(\mathcal{P})$ then we say that $\mathcal{P}$ is {\em directly regular}.
In this case there is an involution $\alpha \in \Aut{\Gamma(\mathcal{P})}$ such that $\alpha(\sigma_1) = \sigma_1^{-1}$, $\alpha(\sigma_2) = \sigma_1^2\sigma_2$ and $\alpha(\sigma_i) = \sigma_i$, for $i=3, \ldots, \rk-1$.

A polytope $\mathcal{P}$ is {\em chiral} if its automorphism group $\Gamma(\mathcal{P})$ has two orbits on the set of flags, and adjacent flags lie in different orbits. 
The automorphism groups of abstract chiral polytopes have been characterised in~\cite{Schulte}.

Following~\cite{Schulte}, the automorphism group of a chiral polytope $\poly$ is generated by some elements $\sigma_{i}$ which satisfy the following relations.
  \begin{align}
\sigma_{i}^{p_i}&=1,&&\text{for $1\leq i\leq \rk-1$},\tag{C1}\label{C1}\\
(\sigma_{i}\cdots\sigma_{j})^2&=1,&&\text{for $1\leq i<j\leq \rk-1$},\label{C2}\tag{C2}
	\end{align}
  where the $p_i$'s are given by the {\em Schl\"{a}fli symbol} $\{p_1,\ldots,p_{\rk-1}\}$ of $\mathcal{P}$.
  For $1\leq i\leq j\leq n-1$, let us define $\sigma_{i,j}:=\sigma_{i}\sigma_{i+1}\cdot\ldots\cdot\sigma_{j}$ and for $0\leq i\leq n$, define $\sigma_{1,i}:=\sigma_{i,n}:=1$.
  We have the subsequent intersection property: for $I,J\subset\set{-1,\ldots,\rk}$,
\begin{align}
\gen{\sigma_{i,j}|i\leq j,i-1,j\in I}\cap\gen{\sigma_{i,j}|i\leq j,i-1,j\in J}=\gen{\sigma_{i,j}|i\leq j,i-1,j\in I\cap J}.
\tag{C3}\label{C3}
\end{align}
In the rank four case, the intersection property is equivalent to
\begin{align}
\gen{\sigma_{1}}\cap\gen{\sigma_{2}}=\{\epsilon\},&&\gen{\sigma_{2}}\cap\gen{\sigma_{3}}=\{\epsilon\}&&\text{and}&&\gen{\sigma_{1},\sigma_{2}}\cap\gen{\sigma_{2},\sigma_{3}}=\gen{\sigma_{2}}.\tag{C3'}\label{C3'}
\end{align}

A group $\Gamma$ is the automorphism group of a chiral polytope if and only if it is generated by at least two elements satisfying the relations (C1), (C2) and (C3), for some finite $p_i\geq 2$, and provided the fact that there is no involutory automorphism $\alpha\in\Aut{\Gamma(\poly)}$ such that $\alpha(\sigma_{1})=\sigma_{1}^{-1}$, $\alpha(\sigma_2)=\sigma_1^2\sigma_2$ and $\alpha(\sigma_{i})=\sigma_{i}$, for all $i=3,\cdots,\rk-1$.
  
  The automorphism group of a face or a coface of a chiral polytope contains a subset of the generating rotations, which satisfies the conditions \eqref{C1} to \eqref{C3}.
  Hence a face or coface of a chiral polytope is a chiral or directly regular polytope and the rotations generate the whole automorphism group or the rotation subgroup respectively (see~\cite[Proposition 4(a)]{Schulte}).
  Since an $(\rk-2)$-face of a chiral polytope contains an automorphism mapping one of its flags to an adjacent flag, it cannot be chiral and must then be directly regular.
  
  \begin{lem}\cite[Proposition 9]{Schulte}\label{lem:rank3facesAreDirReg}
  The $(\rk-2)$-faces and the cofaces of edges of a chiral $\rk$-polytope are directly regular polytopes.
  \label{lem:rank3-faces_are_directly_regular}
\end{lem}

The work done on the existence of a regular polytope for the almost simple groups with socle $\PSL{q}$ will be very useful in helping us to determine if a group is the rotation subgroup of a directly regular polytope or the full automorphism group of a chiral polytope.

\begin{thm}\label{RegRank4}\cite{DiJuTho,LS2005a,LS2008a}
Let $q\geq4$ and $\Autgp$ be an almost simple group with socle $\PSL{q}$.
The group $\Autgp$ is the automorphism group of a regular polytope of rank four if and only if $\Autgp$ is one of $\PGL{5}$, $\PSL{11}$ $\PSL{19}$ or $\Sigma(2,p^{2d})$ (the subgroup of $P\Gamma L(2,q)$ generated by all its $PSL$-involutions and Baer involutions) where $d\geq1$ and $p^{2d}\geq 9$.
\end{thm}
  
\subsection{Projective linear groups}\label{sec:inter}

The projective linear group $\PGL{q}$ is defined for any prime power $q$ as the quotient
\begin{align*}
\PGL{q}:=\GL{q}/\Z{\GL{q}},
\end{align*}
where $\GL{q}=\GL{\GF{q}}$ and $\Z{\GL{q}}$ is the centre of $\GL{q}$.
Moreover, the projective special linear group $\PSL{q}$ is defined as the quotient
\begin{align*}
\PSL{q}:=\SL{q}/\Z{\SL{q}}.
\end{align*}
So we can write each element of $\PGL{q}$ or $\PSL{q}$ as a matrix up to a factor, that is, they are defined up to the multiplication by a scalar matrix of $\GL{q}$ or $\SL{q}$ respectively.
In this paper, each time we write an element of $\PGL{q}$ or $\PSL{q}$ in a matrix form, it is understood that this element is in fact an equivalence class.
One can see this matrix as a representative of the coset of the quotient group.

The groups $\PGL{q}$ and $\PSL{q}$ act via the fractional transformations on the projective line $\PG{q}$ of $q+1$ points.
The stabiliser of a point of the projective line $\PG{q}$ in $\PGL{q}$ is isomorphic to $E_q\rtimes\C{q-1}$.
If $q$ is odd, the stabiliser in $\PSL{q}$ of a point is isomorphic to $E_q\rtimes\C{(q-1)/2}$ and if $q$ is even, then $\PSL{q} = \PGL{q}$.
Moreover, the pointwise stabiliser of two points in $\PGL{q}$ is a cyclic group of order $q-1$ whereas the pointwise stabiliser of two points in $\PSL{q}$ is a cyclic group of order $(q-1)/2$ if $q$ is odd.

We continue this section by the following straightforward lemma.
\begin{lem}\label{lem:order3}
For any prime power $q$, if $\sigma\in\PGL{q}$ is an element of order $3$, then $\sigma\in\PSL{q}$.
\end{lem}

%\subsection{Maximal subgroups}
The maximal subgroups of the projective general linear groups  and projective special linear groups have been classified for example in \cite{MaxSub} (see also~\cite{MR0104735}) where a proof of the following proposition can be found.
\begin{prop}
  Let $q=p^d>3$ be an odd prime power.
  The maximal subgroups of $\PSL{q}$ are given by the following list:
  \begin{itemize}
    %\item $C_p^n\rtimes C_{(q-1)/2}$,
    \item $E_q\rtimes C_{(q-1)/2}$ (the stabilizer of a point of the projective line),
    \item $\D{q-1}$, if $q\geq 13$,
    \item $\D{q+1}$, if $q\not=7,9$,
    \item $A_5$, if $q=p\equiv\pm1\pmod{10}$ or $q=p^2\equiv9\pmod{10}$,%for $q\equiv\pm1\pmod{10}$, where either $q=p$, or $q=p^2$ and $p\equiv\pm3\pmod{10}$,
    \item $A_4$, if $q=p\equiv\pm3,5,\pm13\pmod{40}$,%for $q=p\equiv\pm3\pmod{8}$ and $q\not\equiv\pm1\pmod{10}$,
    \item $S_4$, if $q=p\equiv\pm1\pmod8$,
    \item $\PGL{\sqrt{q}}$, if $d$ is even,
    \item $\PSL{p^e}$, if $d=er$, where $r$ is an odd prime.
  \end{itemize}
%  \label{prop:max_sub_psl}
%\end{prop}
%
%\begin{prop}
  %If $q=p^d>3$ is an odd prime power, the maximal subgroups of $\PGL{q}$ are in given by the following list:
  The maximal subgroups of $\PGL{q}$ are in given by the following list:
  \begin{itemize}
    %\item $C_p^n\rtimes C_{q-1}$,
    \item $E_q\rtimes C_{q-1}$ (the stabilizer of a point of the projective line),
    \item $\D{2(q-1)}$, if $q\not=5$,
    \item $\D{2(q+1)}$,
    \item $S_4$, if $q=p\equiv\pm3\pmod8$,
    \item $\PGL{p^e}$, if $d=er$, where $r$ is an odd prime,
    \item $\PSL{q}$.
  \end{itemize}
  If $q=2^d\geq 4$, the maximal subgroups of $\PSL{q}$ are given by the following list:
  \begin{itemize}
    %\item $C_2^n\rtimes C_{q-1}$
    \item $E_q\rtimes C_{q-1}$ (the stabilizer of a point of the projective line),
    \item $\D{2(q\pm1)}$,
    \item $\PSL{q} = \PGL{2^e}$, if $d=er$, where $r$ is a prime and $e\not=1$.
  \end{itemize}
  \label{prop:max_sub}
 \end{prop}

%\subsection{Subfield subgroups}
From Proposition \ref{prop:max_sub}, one can easily derive the following lemma: 
  \begin{lem}\label{lem:PGLinPSL}
    If $p$ is an odd prime and $\PGL{p^e}\leq\PSL{p^d}$, then $d/e$ is even.
  \end{lem}
%  \begin{proof}    
%    Let $H\cong\PGL{p^f}$ be the largest subfield subgroup such that $\PGL{p^e}\leq H\leq\PSL{p^d}$.
%    Let $K\cong\PSL{p^g}$ be the smallest subfield subgroup such that $\PGL{p^e}\leq H\leq K\leq\PSL{p^d}$.
%    Since $H$ is a maximal subgroup of $K$, $g/f=2$ by Proposition \ref{prop:max_sub_psl}.
%    So $d/e=(d/g)(g/f)(f/e)=2(d/g)(f/e)$ which is even.
%  \end{proof}
	
  The existence of a subgroup $H$ isomorphic to $\PGL{p^f}$ such that $\PSL{p^e}\leq H\leq\PSL{p^d}$ can be guaranteed simply by looking at the parity of $d/f$, as stated in the following Lemma:
  \begin{lem}\label{lem:PSLinPSL}
	If $\PSL{p^e}\leq\PSL{p^d}$ and $d/e$ is even, then there exists $H\cong\PGL{p^f}$ such that $\PSL{p^e}\leq H\leq\PSL{p^d}$.
  \end{lem}
  \begin{proof}
	Suppose there is no such subgroup $H\cong\PGL{p^f}$.
	Consider a chain $\PSL{p^e}=H_0\leq H_1\leq\cdots\leq H_k\leq H_{k+1}=\PSL{p^d}$ where for $i=0,\ldots,k$, each $H_i\cong\PSL{p^{f_i}}$ is maximal in $H_{i+1}\cong\PSL{p^{f_{i+1}}}$.
	Then $f_{i+1}/f_i$ is odd and $d/e=(d/f_{k+1})(f_{k+1}/f_k)\cdots(f_1/e)$ is a product of odd integers, which is odd too, a contradiction.
  \end{proof}
  
  Throughout this paper, the greatest common divisor of two integers $e$ and $f$ will be denoted by $\Gcd{e,f}$.
  
  \begin{prop}\label{prop:interSubf}
Let $H_1\cong\PSL{p^{e_1}}$ or $\PGL{p^{e_1}}$ and $H_2\cong\PSL{p^{e_2}}$ or $\PGL{p^{e_2}}$ be subfield subgroups of $\PSL{p^d}$.
For $i=1,2$, let us denote $H_i'$ the unique subgroup of $H_i$ isomorphic to $\PSL{p^{e_i}}$.
%For $i=1,2$, let $L_i\cong\PSL{p^{e_i}}\leq H_i$.
If $H_1'\cap H_2'$ contains a dihedral subgroup of order $2k$ with $(k,p)\not=(2,2)$, then $H_1\cap H_2$ is a subfield subgroup.
More precisely,
\begin{enumerate}
\item if $H_1\cong\PSL{p^{e_1}}$ and $H_2\cong\PSL{p^{e_2}}$, then 
\begin{align*}
H_1\cap H_2\cong\PSL{p^{\Gcd{e_1,e_2}}};
\end{align*}
\item if $H_1\cong\PGL{p^{e_1}}$ and $H_2\cong\PGL{p^{e_2}}$, then 
\begin{align*}
H_1\cap H_2\cong\PGL{p^{\Gcd{e_1,e_2}}};
\end{align*}
\item if $H_1\cong\PSL{p^{e_1}}$ and $H_2\cong\PGL{p^{e_2}}$, then 
\begin{align*}
H_1\cap H_2\cong\begin{cases}\PGL{p^{\Gcd{e_1,e_2}}}&\text{if $\frac{e_1}{\Gcd{e_1,e_2}}$ is even}\\
							\PSL{p^{\Gcd{e_1,e_2}}}&\text{otherwise}\end{cases}
\end{align*}
\end{enumerate}
\end{prop}
\begin{proof}
Suppose $D\cong D_{2k}$ is contained in the intersection $H_1'\cap H_2'$.
Let $\tilde{e}:=\Gcd{e_1,e_2}$.
There exist $K_1$ and $K_2$, both isomorphic to $\PSL{p^{\tilde{e}}}$ such that $K_1\leq H_1'$ and $K_2\leq H_2'$.
It can be shown, using the Euclidian algorithm, that $\Gcd{p^{e_1}\pm1,p^{e_2}\pm1}\in\set{p^{\Gcd{e_1,e_2}}+1,p^{\Gcd{e_1,e_2}}-1}$.
Hence $D$ is a subgroup of $K_1$ and $K_2$.
Taking $f=\tilde{e}$, we have $f\mid e_1$, $f\mid e_2$, $f\mid d$ and $k\mid \frac{p^f\pm1}{2}$.
By \cite[Lemma 5.7]{JL2013}, there exists a unique $\PSL{p^f}$ such that
\begin{align*}
D&\leq\PSL{p^f}\leq \PSL{p^{e_1}},\\
D&\leq\PSL{p^f}\leq \PSL{p^{e_2}}\text{ and}\\
D&\leq\PSL{p^f}\leq\PSL{p^d}.
\end{align*}
Hence $K_1=K_2$ so that $\PSL{p^{\tilde{e}}}\cong K_1=K_2\leq H_1\cap H_2$.
If $H_1\cong\PSL{p^{e_1}}$ and $H_2\cong\PSL{p^{e_2}}$, then clearly $H_1\cap H_2\cong\PSL{p^{\tilde{e}}}$.
If $H_1\cong\PGL{p^{e_1}}$ and $H_2\cong\PGL{p^{e_2}}$, then $\PSL{p^{\tilde{e}}}\leq\PGL{p^{\tilde{e}}}\leq H_1\cap H_2$ so that $H_1\cap H_2\cong\PGL{p^{\tilde{e}}}$.
If $H_1\cong\PSL{p^{e_1}}$ and $H_2\cong\PGL{p^{e_2}}$ and $e_1/\tilde{e}$ is even then $\PGL{p^{\tilde{e}}}\leq\PSL{p^{e_1}}\cap\PGL{p^{e_2}}$ so $H_1\cong H_2\cong\PGL{p^{\tilde{e}}}$.
If $e_1/\tilde{e}$ is odd then $\PGL{p^{\tilde{e}}}$ is not a subgroup of $\PSL{p^{e_1}}$ so that $H_1\cap H_2\cong\PSL{p^{\tilde{e}}}$.
\end{proof}

\section{Rank $\geq5$}\label{rank5}

In this section, we prove part (1) of Theorem~\ref{maintheo}, that is, that neither $\PSL{q}$ nor $\PGL{q}$ are automorphism groups of chiral polytopes of rank at least five.
\begin{thm}\label{thm:rank5}
  If $q>3$ is a prime power, neither $\PSL{q}$ nor $\PGL{q}$ is the automorphism group of a chiral polytope of rank $\geq5$.
\end{thm}
\begin{proof}
  Suppose $\Autgp\cong\PSL{q}$ or $\Autgp\cong\PGL{q}$ is the automorphism group of a chiral polytope $\mathcal{P}$ of rank $5$ with distinguished generators $\set{\s{1},\s{2},\s{3},\s{4}}$.  
  %Observe that $\gen{\s{1},\s{2},\s{2}\s{3}\s{4}}=\gen{\s{1},\s{2},\s{3}\s{4}}$ and $\gen{\s{1}\s{2}\s{3},\s{3},\s{4}}=\gen{\s{1}\s{2},\s{3},\s{4}}$.
  %By Lemmas~\ref{lem:rank3facesAreDirReg} and \ref{lem:rank5lem}, $\gen{\s{1},\s{2},\s{3}\s{4}}$ and $\gen{\s{1}\s{2},\s{3},\s{4}}$ are the automorphism group of two distinct faces of $\mathcal{P}$ that are regular polyhedra.
  %By \ref{lem:rank3-faces_are_directly_regular}, $\gen{\s{1},\s{2}}$ is the rotations subgroup of a regular polyhedron.
  %The automorphism $\s{3}\s{4}$ is a reflection of this regular polyhedron so $\gen{\s{1},\s{2},\s{3}\s{4}}$ is the automorphism group of a directly regular polyhedron.
  %Similarly for $\gen{\s{1}\s{2},\s{3},\s{4}}$.
  By Lemma \ref{lem:rank3facesAreDirReg}, $\gen{\s{1}\s{2},\s{3},\s{4}}$ and $\gen{\s{1},\s{2},\s{3}\s{4}}$ are two distinct subgroups of $\PGL{q}$, each of them being the automorphism group of a directly regular polyhedron.
  According to the classification of the maximal subgroups of $\PGL{q}$, $\gen{\s{1},\s{2},\s{3}\s{4}}$ and $\gen{\s{1}\s{2},\s{3},\s{4}}$ must be isomorphic to one of $\Sym{4}$
  %$\Alt{5}$, 
  %$\Alt{4}$, 
  %$\PSL{q_0}$ 
  or $\PGL{p^e}$ for $d=er$ where $r>1$.
The other subgroups cannot be automorphism groups of directly regular polyhedra.
The only directly regular polyhedron for $\Sym{4}$ has type $\set{3,3}$.
This already implies that both $\gen{\s{1},\s{2},\s{3}\s{4}}$ and $\gen{\s{1}\s{2},\s{3},\s{4}}$ cannot be isomorphic to $S_4$ for otherwise, we obtain the 4-simplex which is not a chiral polytope.
Furthermore, if $q=2^e$, then $\PGL{2^{e}}=\PSL{2^e}$ so the $3$-faces can only have $S_4$ as automorphism group, which is impossible.
Hence $\PSL{2^d}$ is not the automorphism group of a chiral $5$-polytope.
So we may assume $p$ odd.

 We start by considering the case where
 \begin{align}
 \Autgp\cong\PSL{p^d}&&\text{and}&&\gen{\s{1},\s{2},\s{3}\s{4}}\cong\PGL{p^{e}},\label{eq:PGLe}
 \end{align}
 where $e$ is a divisor of $d$ (see Figure \ref{fig:1} for a related partial diagram of the subgroup structure of $\PSL{p^d}$).
 First note that, by Lemma \ref{lem:PGLinPSL}, $d/e$ is even and hence $d$ is even.
 The subgroup $\gen{\s{1},\s{2}}$, being the rotation subgroup of $\gen{\s{1},\s{2},\s{3}\s{4}}$ must be isomorphic to $\PSL{p^{e}}$.
 Consider the subgroup $\gen{\s{1},\s{2},\s{3}}$.
 It must be a subfield subgroup of $\Autgp$ containing $\PSL{p^{e}}$.
 If $\gen{\s{1},\s{2},\s{3}}\cong\PGL{p^{f}}$ for some multiple $f$ of $e$, then, by the intersection property (C3), $\PGL{p^e}\cap\PGL{p^f}\cong\gen{\s{1},\s{2},\s{3}\s{4}}\cap\gen{\s{1},\s{2},\s{3}}=\gen{\s{1},\s{2}}\cong\PSL{p^{e}}$  (for, otherwise, $\langle \alpha,\beta,\gamma,\delta\rangle \leq \langle\alpha,\beta,\gamma\rangle$ and therefore $\Gamma = \langle\alpha,\beta,\gamma\rangle$). But this contradicts Proposition \ref{prop:interSubf}.
 Hence
 \begin{align}
 \gen{\s{1},\s{2},\s{3}}\cong\PSL{p^{f}}.\label{eq:PGLf}
 \end{align}
 The quotient $f/e$ must be odd, for otherwise, by Proposition \ref{prop:interSubf}, $\PSL{p^{f}}\cap\PGL{p^{e}}\cong\PGL{p^{e}}$.
 Therefore, $(d/e)/(f/e)=d/f$ is even and $\PGL{p^f}$ is a subgroup of $\PSL{p^d}$.
 By Lemma \ref{lem:PSLinPSL}, there exists $H\cong\PGL{p^g}$ such that the intersection $\PGL{p^{f}}\cap\PGL{p^{e}}$ contains $\PSL{p^{e}}$ so,
 \begin{align}
 \PGL{p^{f}}\cap\PGL{p^{e}}\cong\PGL{p^{e}}\label{eq:PGLpe}
 \end{align}
 which means that $\PGL{p^{e}}$ is a subgroup of $\PGL{p^{f}}$.
 Hence, putting \eqref{eq:PGLe}, \eqref{eq:PGLf} and \eqref{eq:PGLpe} together, we have
 \begin{align*}
 \PSL{p^d}=\gen{\s{1},\s{2},\s{3},\s{4}}=\gen{\PSL{p^{f}},\PGL{p^{e}}}\leq\PGL{p^{f}},
 \end{align*}
 which is impossible since $f<d$.
 The symmetric case where $\gen{\s{1}\s{2},\s{3},\s{4}}\cong\PGL{p^e}$ is similar to deal with.
 This already shows that $\PSL{q}$ is not the automorphism group of a chiral $5$-polytope.
 
 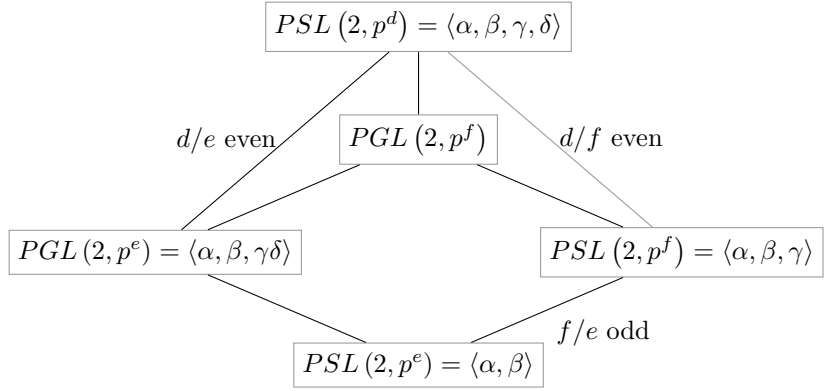
\begin{figure}
  \begin{tikzpicture}

  \pgfmathsetmacro{\bx}{5}
  \pgfmathsetmacro{\by}{-1.5}
  \pgfmathsetmacro{\cx}{10}
  \pgfmathsetmacro{\cy}{-1.5}
  \pgfmathsetmacro{\dx}{0}
  \pgfmathsetmacro{\dy}{0}
  \pgfmathsetmacro{\r}{1.5}
  \pgfmathsetmacro{\debut}{180}
  
  \node[draw=black!40] (a1) at (0,0) {$\PSL{p^d}=\gen{\s{1},\s{2},\s{3},\s{4}}$};
  \node[draw=black!40] (a2) at (-3.5,-3) {$\PGL{p^e}=\gen{\s{1},\s{2},\s{3}\s{4}}$};
  \node[draw=black!40] (a3) at (0,-4.5) {$\PSL{p^e}=\gen{\s{1},\s{2}}$};
  
  \node[draw=black!40] (a4) at (3.5,-3) {$\PSL{p^f}=\gen{\s{1},\s{2},\s{3}}$};
  \node[draw=black!40] (a5) at (0,-1.5) {$\PGL{p^f}$};
  
  \path[draw=black!40] (a1) edge [left] node {$d/e$ even} (a2);
  \path[draw=black!40] (a2) edge (a3);
  \path[draw=black!40] (a3) edge [below right] node {$f/e$ odd} (a4);
  \path[draw=black!40] (a4) edge (a5);
  \path[draw=black!40] (a1) edge [draw=black!40] node [right] {$d/f$ even} (a4);
    \path[draw=black!40] (a5) edge (a1);
    \path[draw=black!40] (a2) edge (a5);

 \end{tikzpicture}
 \caption{A partial diagram of the subgroup structure of $\PSL{p^d}$}
 \label{fig:1}
 \end{figure}
 
 Now assume that $\Autgp\cong\PGL{p^d}$ and $\gen{\s{1},\s{2},\s{3}\s{4}}\cong\PGL{p^{e_1}}$ (see Figure \ref{fig:2} for a related diagram).
 Then we have two cases: either $\gen{\s{1}\s{2},\s{3},\s{4}}$ is isomorphic to $\PGL{p^{e_2}}$ or to $S_4$.
 First suppose $\gen{\s{1}\s{2},\s{3},\s{4}}\cong\PGL{p^{e_2}}$.
 As before, $\gen{\s{1},\s{2}}$ and $\gen{\s{3},\s{4}}$ are the rotation subgroups of $\PGL{p^{e_1}}$ and $\PGL{p^{e_2}}$ respectively, so
 \begin{align*}
 \gen{\s{1},\s{2}}&\cong\PSL{p^{e_1}}\text{ and}\\
 \gen{\s{3},\s{4}}&\cong\PSL{p^{e_2}}.
 \end{align*}
 Each of these subgroups is contained in a subgroup isomorphic to $\PSL{p^d}$.
 But $\PSL{p^d}$ is unique in $\Autgp$ so
 \begin{align*}
 \PGL{p^d}=\gen{\s{1},\s{2},\s{3},\s{4}}=\gen{\PSL{p^{e_1}},\PSL{p^{e_2}}}\leq\PSL{p^d},
 \end{align*}
 a contradiction.
 Finally suppose $\Autgp\cong\PGL{p^d}$, $\gen{\s{1}\s{2},\s{3},\s{4}}\cong\PGL{p^{e_1}}$ and $\gen{\s{1}\s{2},\s{3},\s{4}}\cong S_4$.
 The subgroup $\gen{\s{3},\s{4}}$ is isomorphic to $A_4$.
 Since $\s{3}$ and $\s{4}$ have order $\order{\s{3}}=\order{\s{4}}=3$, by Lemma \ref{lem:order3}, $\gen{\s{3},\s{4}}$ is a subgroup of $\PSL{p^d}$.
 Hence
 \begin{align*}
 \PGL{p^d}=\gen{\s{1},\s{2},\s{3},\s{4}}=\gen{\PSL{p^{e_1}},A_4}\leq\PSL{p^d},
 \end{align*}
 which leads to the same contradiction as before.
 
 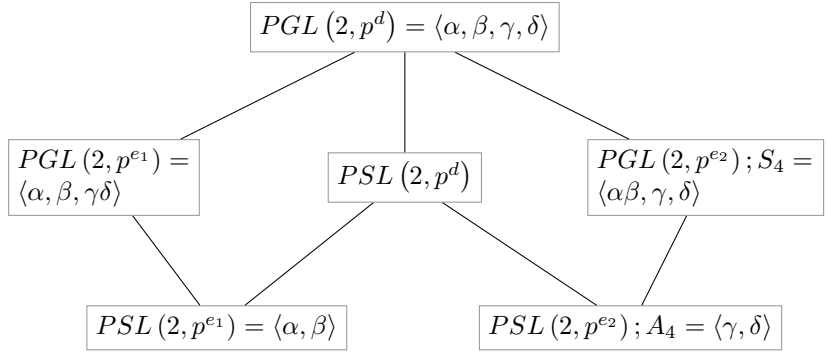
\begin{figure}
  \begin{tikzpicture}

  \pgfmathsetmacro{\bx}{5}
  \pgfmathsetmacro{\by}{-1.5}
  \pgfmathsetmacro{\cx}{10}
  \pgfmathsetmacro{\cy}{-1.5}
  \pgfmathsetmacro{\dx}{0}
  \pgfmathsetmacro{\dy}{0}
  \pgfmathsetmacro{\r}{1.5}
  \pgfmathsetmacro{\debut}{180}
  
  \node[draw=black!40] (a1) at (0,-2) {$\PGL{p^d}=\gen{\s{1},\s{2},\s{3},\s{4}}$};
  \node[draw=black!40,text width=2.3cm] (a2) at (-4,-4) {$\PGL{p^{e_1}}=\gen{\s{1},\s{2},\s{3}\s{4}}$};
  \node[draw=black!40] (a3) at (-2.5,-6) {$\PSL{p^{e_1}}=\gen{\s{1},\s{2}}$};
  
  \node[draw=black!40,text width=2.9cm] (a4) at (4,-4) {$\PGL{p^{e_2}};S_4=\gen{\s{1}\s{2},\s{3},\s{4}}$};
  \node[draw=black!40] (a5) at (3,-6) {$\PSL{p^{e_2}};A_4=\gen{\s{3},\s{4}}$};
  
  \node[draw=black!40] (a6) at (0,-4) {$\PSL{p^d}$};
  
  \path[draw=black!40] (a1) edge (a2);
  \path[draw=black!40] (a2) edge (a3);
  \path[draw=black!40] (a3) edge (a6);
  \path[draw=black!40] (a1) edge (a4);
    \path[draw=black!40] (a4) edge (a5);
    \path[draw=black!40] (a5) edge (a6);
        \path[draw=black!40] (a6) edge (a1);

 \end{tikzpicture}
 \caption{A partial diagram of the subgroup structure of $\PGL{p^d}$.}
 \label{fig:2}
 \end{figure}
 
 The case where $\Autgp\cong\PGL{p^d}$ and $\gen{\s{1}\s{2},\s{3},\s{4}}\cong\PGL{p^e}$ is similar to deal with and can therefore be ruled out.
Hence $\Autgp\in\set{\PSL{q},\PGL{q}}$ cannot be the automorphism group of a rank-$5$ chiral polytope.

  It remains to investigate the rank $>5$.
  The $n-1$ faces of a chiral polytope of rank $n>5$ with automorphism group isomorphic to a projective linear group are directly regular or chiral polytopes of rank $n-1>4$ with automorphism group isomorphic to a projective linear group.
  So by Theorem \ref{RegRank4}, there is no chiral polytope of rank $n>4$ for $\PSL{q}$ and $\PGL{q}$.
  \end{proof}
  
  This proves part \eqref{main:rg5} of Theorem \ref{maintheo}.

\section{Rank $4$ and $\PGL{q}$}\label{sec:PGL4}

In this section we prove parts \eqref{main:PGL} of Theorem~\ref{maintheo} by constructing, for each prime power $q$, chiral polytopes of rank 4 whose automorphism group is $\PGL{q}$.
We denote the generators of the group $\alpha$, $\beta$ and $\gamma$.
Throughout this paper, we use the following abuse of language.
If each of the parabolic subgroups $\langle \alpha, \beta\rangle$ and $\langle \beta, \gamma \rangle$ of $\gen{\s{1},\s{2},\s{3}}$ fixes at least one point, we say that the polytope has {\em non fixed point free parabolic subgroups}. Otherwise, we say that the polytope has {\em fixed point free parabolic subgroups}.
The goal of this section is to classify the chiral $4$-polytopes with non fixed point free parabolic subgroups and of type $\PGL{q}$, that is whose automorphism group is isomorphic to $\PGL{q}$.

We start with a small and easy technical lemma that will be useful for the classification.
\begin{lem}\label{lem:order}
Let $F$ be a finite field and $g\in F$ be an element of multiplicative order $|g| =: m$.
Then the multiplicative order of $-g$ is
\begin{align*}
\order{-g}=\begin{cases}2m&\text{if }m\equiv1\pmod{2}\\
m&\text{if }m\equiv0\pmod{4}\\
m/2&\text{if }m\equiv2\pmod{4}.
\end{cases}
\end{align*}
\end{lem}
\begin{proof}
If $m\equiv1\pmod{2}$, then $\Gcd{2,m}=1$ so $\order{-g}=\order{-1}\order{g}=2m$.
If $m\equiv0\pmod{2}$, then $\order{g}/2$ divides $\order{-g}$ so $\order{-g}\in\set{m/2,m}$.
If $m\equiv0\pmod{4}$, then $(-g)^{m/2}=g^{m/2}\not=\epsilon$ so $\order{-g}=m$ whereas if $m\equiv2\pmod{4}$, then $(-g)^{m/2}=-g^{m/2}=\epsilon$ so $\order{-g}=m/2$.
\end{proof}

The following theorem classifies the chiral polytopes with non fixed point free parabolic subgroups and automorphism group isomorphic to $\PGL{q}$.

%\begin{thm} \label{thm:PGL}
%%Let $p$ be an odd prime number and $k>4$
%%Let $d$ be the smallest integer such that $k\mid p^d-1$ and $k\nmid(p^d\pm1)/2$ (if it exists).
%%Then, up to isomorphism, $\PGL{q}$ is the automorphism group $\Gamma$ of exactly $\eul{k}/d$ chiral $4$-polytopes type $[k,k,k]$ or $[k/2,k,k/2]$ according as $p^d\not\equiv3\pmod{4}$ or $p^d\equiv3\pmod{4}$, where $\Gamma$ has no fixed point free parabolic subgroups.
%
%Let $q:=p^d>4$ be a prime power and $K$ the set of positive odd integers $k>2$ such that $k\mid p^d-1$, $k\nmid(p^d\pm1)/2$ (if $p^d$ is odd) and $k\nmid p^e\pm1$ for all $e<d$.  Then
%\begin{enumerate}
%\item All chiral $4$-polytopes of type $\PGL{q}$ and with non fixed point free parabolic subgroups have the Schl\"{a}fli type $[k,k,k]$ or $[k/2,k,k/2]$, according as $q\not\equiv3\pmod{4}$ or $q\equiv3\pmod{4}$, for some $k\in K$.
%\item For all $k\in K$, $\PGL{q}$ is the automorphism group of exactly $\eul{k}/d$ chiral $4$-polytopes with non fixed point free parabolic subgroups and of Schl\"{a}fli type $[k,k,k]$ or $[k/2,k,k/2]$ according as $q\not\equiv3\pmod{4}$ or $q\equiv3\pmod{4}$.
%\end{enumerate}
%\end{thm}

\begin{thm} \label{thm:PGL}
Let $q:=p^d>4$ be a prime power.
The only chiral $4$-polytopes with non fixed point free parabolic subgroups and automorphism group $\PGL{q}$ are as follows: for each  integer $k>2$ such that $k\mid q-1$, $k\nmid (q-1)/2$ and $k\nmid p^e\pm1$ for all $e<d$,
\begin{enumerate}
\item if $q\not\equiv3\pmod{4}$, there are $\eul{k}/d$ chiral $4$-polytopes of type $[k,k,k]$.
\item if $q\equiv3\pmod{4}$, there are $\eul{k}/d$ chiral $4$-polytopes of type $[k/2,k,k/2]$.
%\item For all $k\in K$, $\PSL{q}$ is the automorphism group of exactly $\eul{k}/d$ chiral $4$-polytopes with non fixed point free parabolic subgroups and of type $[k,k,k]$ or $[k/2,k,k/2]$ according as $k\equiv0\pmod{4}$ or $k\equiv2\pmod{4}$.
\end{enumerate}
\end{thm}

\begin{proof}
  If $\gen{\s{1},\s{2},\s{3}}$ is the automorphism group of a chiral $4$-polytope of type $\PGL{q}$ and with non fixed point free parabolic subgroups, then $\s{2}$ fixes at least one point.
  If $\s{2}$ fixes only one point, then $\s{1}$ and $\s{3}$ fix the same point as $\s{2}$ and $\gen{\s{1},\s{2},\s{3}}$ is a subgroup of the stabiliser of a point.
  So $\s{2}$ fixes two points and thus the order of $\s{2}$ divides $q-1$.
  By $2$-transitivity of $\PGL{q}$ on the projective line, we may assume that $\s{2}$ fixes $\infty$ and $0$.
  Let $\prim$ be a primitive element of $\GF{q}$.
  Then $\s{2}$ is conjugate to 
  \begin{align*}
    \s{2}':=\left[\begin{matrix}\prim^l&0\\0&1\end{matrix}\right],
  \end{align*}
  for some positive integer $l\leq q-1$.
  So we may assume $\s{2}=\s{2}'$.
  Let $k$ be the order of $j^l$.
  We now show that, for each $\s{2}$, there exist, up to duality, a unique $\s{1}$ and a unique $\s{3}$ such that $\gen{\s{1},\s{2},\s{3}}\cong\PGL{q}$ is the automorphism group, with non fixed point free parabolic subgroups, of a chiral polytope and where $\order{\s{1}}=\order{\s{3}}=k$ or $k/2$ according as to whether $q\not\equiv3\pmod{4}$ or $q\equiv3\pmod{4}$.
  We may assume that $\s{1}$ fixes $\infty$, as if $\s{1}$ fixes $0$, then we obtain the dual polytope.
  So $\s{1}$ is of the form
  \begin{align*}
    \s{1}:=\left[\begin{matrix}\scal{1}&\scal{2}\\0&1\end{matrix}\right],
  \end{align*}
  where $\scal{1}\not=0$ and $\scal{2}$ are scalars of $\GF{q}$.
  To satisfy condition (C2), $\s{1}\s{2}$ must be an involution, so we have $\Tr{\s{1}\s{2}}=j^l\scal{1}+1=0$.
  Equivalently, $\scal{1}=-j^{-l}$.
  Now if $\s{3}$ is an element of $\PGL{q}$ fixing $0$, then it has the form
  \begin{align*}
    \s{3}:=\left[\begin{matrix}1&0\\\scal{3}&\scal{4}\end{matrix}\right],
  \end{align*}
  where $\scal{3}$ and $\scal{4}\not=0$ are scalars of $\GF{q}$.
  Again, $\s{2}\s{3}$ must be an involution, so we have $\Tr{\s{2}\s{3}}=j^l+\scal{4}=0$.
  Equivalently $\scal{4}=-j^{l}$.
  Finally, $\s{1}\s{2}\s{3}$ must be an involution, so $\Tr{\s{1}\s{2}\s{3}}=-1+\scal{2}\scal{3}-j^l=0$.
  Hence $\scal{3}=\scal{2}^{-1}(1+j^{l})$ and for each $\scal{2}\in\GF{q}$, we obtain the triple
  \begin{align}\label{eq:triple}
  \s{1}=\left[\begin{matrix}
  -j^{-l}&\scal{2}\\0&1
  \end{matrix}\right],&&
    \s{2}=\left[\begin{matrix}
  j^l&0\\0&1
  \end{matrix}\right],&&
    \s{3}=\left[\begin{matrix}
  1&0\\\scal{2}^{-1}(1+j^l)&-j^l
  \end{matrix}\right].
  \end{align}
  Observe that, for a fixed $\scal{2}$, the element $\tau$ defined by
  \begin{align*}
  \tau:=\left[
 \begin{matrix}
 1&0\\0&\scal{2}
 \end{matrix}\right]
  \end{align*}
  defines an isomorphism between the triple in \eqref{eq:triple} and the following triple: 
  \begin{align}\label{eq:triple2}
  \s{1}=\left[\begin{matrix}
  -j^{-l}&1\\0&1
  \end{matrix}\right],&&
    \s{2}=\left[\begin{matrix}
  j^l&0\\0&1
  \end{matrix}\right],&&
    \s{3}=\left[\begin{matrix}
  1&0\\1+j^l&-j^l
  \end{matrix}\right],
  \end{align}
  Hence, for all $\scal{2}$, the triples defined in \eqref{eq:triple} are isomorphic and we may assume $\scal{2}=1$.
  %Moreover, for each $i\leq d$, the $i$-composition of the Frobenius automorphism $\varphi^i$ induces an isomorphism between the triple $\s{1}$, $\s{2}$, $\s{3}$ and the triple $\s{1}^{p^i}$, $\s{2}^{p^i}$, $\s{3}^{p^i}$.
  In a cyclic subgroup of order $k$ of $\PGL{q}$ there are $\eul{k}$ elements of order $k$ and two elements are conjugate if and only if they belong to the same orbit under the action of the Frobenius automorphism.
  Since each of those orbits has length $d$, there are $\eul{k}/d$ non isomorphic triples $[\s{1},\s{2},\s{3}]$ defined as in \eqref{eq:triple2}, where $l\leq q-1$ such that $j^l$ has order $k$ and $j^{l'p^r}\not=j^l$ for all $r\in\set{0,\ldots,d-1}$ and for all $l'<l$ such that $j^{l'}$ has order $k$.
  
  We now show that each of those triples generates the automorphism group of a chiral polytope of type $\PGL{q}$.
  
  We start by computing the orders of $\s{1}$ and $\s{3}$.
 Define $\rho$ by
  \begin{align*}
  \rho:=\left[\begin{matrix}
  0&j^l\\1+j^l&0
  \end{matrix}\right].
  \end{align*}
  Then the conjugate of $\s{1}^{-1}$ by $\rho$ is equal to $\s{3}$ so that $\order{\s{3}}=\order{\s{1}}$.
  
  The $k$-power of $\s{1}$ is given by
  \begin{align}\label{eq:s1k}
  \s{1}^k=\left[\begin{matrix}
  (-j^{-l})^k&\sum_{i=0}^{k-1}(-j^{-l})^i\\0&1
  \end{matrix}\right]
\end{align}   
  
  If $q\equiv0\pmod{2}$, then $\order{-(j^{-l})}=\order{j^{-l}}=\order{j}=k$ so $\order{\s{1}}=\order{\s{3}}=k$.
  If $q\equiv1\pmod{2}$, then either $q\equiv1\pmod{4}$, in which case $k\equiv0\pmod{4}$, or $q\equiv3\pmod{4}$ in which case $k\equiv2\pmod{4}$.
  So, by Lemma \ref{lem:order}, if $q\equiv1\pmod{4}$, then $\order{-j^{-l}}=k$ and if $q\equiv3\pmod{4}$, then $\order{-j^{-l}}=k/2$.
  Hence $\s{1}$ has order $k$ or $k/2$ according as $q\not\equiv3\pmod{4}$ or $q\equiv3\pmod{4}$.
    
%  Note that if $q\equiv3\pmod{4}$, then $k\equiv2\pmod{4}$ and for all $i\in\set{0,\ldots,k-1}$, $(-j^{-l})^{i+k/2}=(-j^{-l})^i$ in which case $\sum_{i=0}^{k-1}(-j^{-l})^i=2\sum_{i=0}^{k/2-1}(-j^{-l})^i$.
%  Hence if $q\not\equiv3\pmod{4}$, the set $\set{(-j^{-l})^i\mid i\in\set{0,\ldots,k-1}}$ is the set of roots of the polynomial $\xi^k-1$ and if $q\equiv3\pmod{4}$, the set $\set{(-j^{-l})^i\mid i\in\set{0,\ldots,k/2-1}}$ is the set of roots of the polynomial $\xi^{k/2}-1$.
%  Hence the sum in \eqref{eq:s1k} is equal to zero and $\s{1}$ has order $k$ or $k/2$ according as $q\not\equiv3\pmod{4}$ or $q\equiv3\pmod{4}$.
%  \begin{align*}
%  \s{1}^k=\left[\begin{matrix}
%  (-j^{-l})^k&0\\0&1
%  \end{matrix}\right],
%  \end{align*}

%	Observe that $\gen{\s{1},\s{2}}=\gen{\s{1}\s{2},\s{2}}$.
	Let $w\in\gen{\s{1},\s{2}}$.
	Then $w=\s{2}^{\kappa_1}(\s{1}\s{2})\s{2}^{\kappa_2}(\s{1}\s{2})\cdots(\s{1}\s{2})\s{2}^{\kappa_n}$.
	So for any point $z$ in the affine line, $w(z)=\epsilon j^{(\kappa_1+\cdots+\kappa_n)l}z+A$, for some scalar $A$ and $\epsilon\in\set{-1,1}$.
	The trace of $w$ is given by $\Tr{w}=\epsilon j^{(\kappa_1+\cdots+\kappa_n)l}+1$ in which case $\order{w}\leq\order{j^l}=\order{\s{2}}$.
	It follows that $\gen{\s{1},\s{2}}\cong E_q\rtimes C_{\order{\s{2}}}$.
	Similarly, $\gen{\s{2},\s{3}}\cong E_q\rtimes C_{\order{\s{2}}}$ and thus $\gen{\s{1},\s{2}}\cap\gen{\s{2},\s{3}}\cong C_{\order{\s{2}}}$ so the intersection property holds.

  %Since $\order{\s{1}}=\order{\s{3}}$ and $\order{\s{2}}$ is a multiple of $\order{\s{1}}$, the intersection property holds, by Lemma \ref{lem:inter}.
  
  We now show that $\s{1}$, $\s{2}$ and $\s{3}$ generate $\PGL{q}$.
  By hypothesis, $k\nmid(q-1)/2$ and $k\nmid(p^e\pm1)$ for all $e<d$, so $\gen{\s{1},\s{2},\s{3}}$ is not contained in a proper subfield subgroup.
  Moreover, $\gen{\s{1},\s{2},\s{3}}$ is fixed point free and clearly not dihedral.
  Since $k>4$, $\gen{\s{1},\s{2},\s{3}}$ is not an icosahedral subgroup, so $\gen{\s{1},\s{2},\s{3}}\cong\PGL{q}$.
  Note that if $k\mid (q-1)/2$ or if there exists $e<d$ such that $k\mid p^e\pm1$, then $\gen{\s{1},\s{2},\s{3}}$ is isomorphic to a proper subfield subgroup of $\PGL{q}$.
  
  It remains to show that $\gen{\s{1},\s{2},\s{3}}$ is the automorphism group of a chiral polytope $\poly$.
  This is straightforward since if we assume that $\langle \alpha,\beta\rangle$ is a subgroup of $\AGL{q}$ that is the automorphism group of a regular polytope of rank at least 3, then $\AGL{q}$  has elementary abelian groups of order $4$. Hence $q$ is even and $k = 2$, contradicting the fact that $k>2$.
   Hence the facets of $\poly$ are chiral and so $\poly$ is chiral.
  \end{proof}
  
  Note that $\PGL{4}\cong A_5$, $\PGL{3}\cong S_3$ and $\PGL{2}\cong C_2$ are not automorphism groups of chiral polytopes.
  This observation and the previous result yield the following corollary.
  
  \begin{cor}
  Let $q=p^d$ be a prime power.
  Then $\PGL{q}$ is the automorphism group of a chiral $4$-polytope if and only if $q>4$.
  \end{cor}
  \begin{proof}
  For each prime power $q>4$, take $k:=q-1$ which satisfies the assumptions of Theorem \ref{thm:PGL}.
  So for each $q>4$, there exists at least one chiral $4$-polytope with $\PGL{q}$ as automorphism group.
  \end{proof}
  
  This proves part \eqref{main:PGL} of Theorem \ref{maintheo}.

\section{Rank $4$ and $\PSL{q}$}\label{sec:PSL4}
As in the previous section, we denote the generators of the group $\Gamma$ by $\alpha$, $\beta$ and $\gamma$.

In \cite{JL2013} and \cite{JLM2011}, the authors proved that there are infinitely many $q$ for which $\PSL{q}$ is a smooth quotient of the $[5,3,5]^+$ or $[3,5,3]^+$ Coxeter even group.
For each characteristic $p$, they showed that $p^4$ is the maximum power for which $\PSL{p^d}$ is such a quotient.
More precisely, they showed that $p$, $p^2$ and $p^4$ are the only possible prime powers.
Hence if $q\equiv3\pmod{4}$, $\PSL{q}$ is the automorphism group of a chiral $4$-polytope of type $[3,5,3]$ or $[5,3,5]$ only when $q=p$.
Such a quotient is the automorphism group of a chiral polytope if it verifies the conditions \eqref{C1} to \eqref{C3} and is not the rotation subgroup of a regular polytope.
The purpose of this section is to show that the smooth quotients obtained in \cite{JLM2011} and \cite{JL2013} are indeed chiral $4$-polytopes.
We obtain a similar result for smooth quotients of the Coxeter even subgroup $[5,3,4]^+$ (and its dual $[4,3,5]^+$).
More remarkable, we show that, if $q=p\equiv3\pmod{4}$, each chiral polytope of rank $4$ with automorphism group isomorphic to $\PSL{q}$ is, up to duality, one of those with type $[5,3,5]$, $[3,5,3]$ or $[5,3,4]$.
As in the previous section, we also classify the chiral polytopes with non fixed point free parabolic subgroups and automorphism group isomorphic to $\PSL{q}$.

\subsection{Case $q\equiv1\pmod{4}$}

As in Section~\ref{sec:PGL4}, we say that a chiral 4-polytope with non fixed point free parabolic subgroups is a polytope whose automorphism group, generated by $\alpha, \beta$ and $\gamma$, is such that both $\langle \alpha,\beta\rangle$ and $\langle \beta,\gamma\rangle$ fix at least one point of $PG(1,q)$.
The next theorem classifies such polytopes with automorphism group $\PSL{q}$.

%\begin{thm} \label{thm:PSL}
%Let $q:=p^d\equiv1\pmod{4}$ be an odd prime power and let $K$ be the set of integers $k>2$ such that $k\mid (q-1)/2$ and $k\nmid p^e\pm1$ for all $e<d$.
%\begin{enumerate}
%\item All chiral $4$-polytopes of type $\PSL{q}$, with non fixed point free parabolic subgroups have the Schl\"{a}fli type $[k,k,k]$ or $[k/2,k,k/2]$ for some $k\in K$.
%\item For all $k\in K$, $\PSL{q}$ is the automorphism group of exactly $\eul{k}/d$ chiral $4$-polytopes with non fixed point free parabolic subgroups and of type $[k,k,k]$ or $[k/2,k,k/2]$ according as $k\equiv0\pmod{4}$ or $k\equiv2\pmod{4}$.
%\end{enumerate}
%\end{thm}

\begin{thm} \label{thm:PSL}
Let $q:=p^d\equiv1\pmod{4}$ be a prime power.
The only chiral $4$-polytopes with non fixed point free parabolic subgroups and automorphism group isomorphic to $\PSL{q}$ are the following: for each integer $k>2$ such that $k\mid (q-1)/2$ and $k\nmid p^e\pm1$ for all $e<d$,
\begin{enumerate}
\item if $k\equiv0\pmod{4}$, there are $\eul{k}/d$ chiral $4$-polytopes of type $[k,k,k]$.
\item if $k\equiv2\pmod{4}$, there are $\eul{k}/d$ chiral $4$-polytopes of type $[k/2,k,k/2]$.
%\item For all $k\in K$, $\PSL{q}$ is the automorphism group of exactly $\eul{k}/d$ chiral $4$-polytopes with non fixed point free parabolic subgroups and of type $[k,k,k]$ or $[k/2,k,k/2]$ according as $k\equiv0\pmod{4}$ or $k\equiv2\pmod{4}$.
\end{enumerate}
\end{thm}

\begin{proof}
This is similar to the proof of Theorem \ref{thm:PGL} except that, even though $q\equiv1\pmod{4}$, $k$ is not necessarily congruent to $0\pmod{4}$.

Suppose $k$ is odd.
Since $\gen{\s{1},\s{2}}$ is not contained in a subfield subgroup, it contains a subgroup $E_q$ and therefore has exactly two orbits on the projective line.
%\begin{commentJ}
%Add that if $E_q\not\leq\gen{\s{1},\s{2}}$, then $\order{\s{1}}$ and $\order{\s{2}}$ divide $p^e-1$, for some $e<d$, so $E_q\leq\gen{\s{1},\s{2}}$?
%Equivalently, if $E_{p^e}:C_m\leq\AGL{p^d}$, with $e<d$, then $m|p^f-1$, for some $f<d$.
%\end{commentJ}
%{\color{orange}
%\begin{conj}
%$H$ is a subgroup of $\AGL{p^n}$ if and only if $H\cong E_{p^e}:C_k$ where $e\in\set{0,\cdots,n}$ and $k|%p^e-1$ and $k|p^n-1$.
%\end{conj}} 
Hence $\s{1}$ is conjugate to an element fixing $0$ and $\infty$.
Similarly, $\s{3}$ is conjugate to an element fixing $0$ and $\infty$.
But then we have two subgroups $E_q:C_{2k}$ having a cyclic group $C_k$ in their intersection. 
The normalisers $N_{\PSL{q}}(C_k) = N_{\PSL{q}}(C_{2k}) = D_{q-1}$. Also we have $N_{\PSL{q}}(E_q:C_{2k}) = N_{\PSL{q}}(E_q:C_{k}) = E_q:C_{(q-1)/2}$. This implies the two subgroups $E_q:C_{2k}$ must have at least a subgroup $C_{2k}$ in their intersection and hence
the intersection property fails.

So $k$ must be even.
Now the computation for the order of $\s{1}$ is the same as in the proof of Theorem \ref{thm:PGL} and $\order{\s{1}}=k$ or $k/2$ according as $k\equiv0\pmod{4}$ or $k\equiv2\pmod{4}$.

It remains to show that the triple $[\s{1},\s{2},\s{3}]$ generates $\PSL{q}$.
This is straightforward since $k\mid(q-1)/2$ and for all $e<d$, $k$ does not divide $p^e\pm1$.
So $\gen{\s{1},\s{2},\s{3}}$ is not contained in a subfield subgroup and is isomorphic to $\PSL{q}$.
Note that if there exists $e<d$ such that $k\mid p^e\pm1$, then $\gen{\s{1},\s{2},\s{3}}$ is the automorphism group of a proper subfield subgroup of $\PSL{q}$.
  \end{proof}
  
  Observe that $\PSL{9}$ and $\PSL{5}\cong A_5$ are not automorphism groups of chiral $4$-polytopes.
  This observation and the previous result yield the following corollary.
  
  \begin{cor}
  Let $q:=p^d\equiv1\pmod{4}$ be a prime power.
  Then $\PSL{q}$ is the automorphism group of a chiral $4$-polytope if and only if $q>9$.
  \end{cor}
  \begin{proof}
  For each $9<q\equiv1\pmod{4}$, take $k:=(q-1)/2$ which satisfies the assumptions of Theorem \ref{thm:PSL} and so there exists a chiral $4$-polytope of type $\PSL{q}$.
  The result follows then by the above observation.
  \end{proof}
  
  This proves part \eqref{PSL_1mod4} of Theorem \ref{maintheo}.

\subsection{Case $q\equiv3\pmod{4}$}

Again here, $\alpha, \beta$ and $\gamma$ are the generators of the automorphism group of the chiral 4-polytopes we are looking at.

\begin{lem}\label{lem:FixedPointFree}
Let $q\equiv3\pmod{4}$ be a prime power.
Both subgroups $\langle \alpha,\beta\rangle$ and $\langle\beta,\gamma\rangle$ do not fix a point.
\end{lem}
\begin{proof}
This is straightforward since any involution of $\PSL{q}$ is fixed point free whenever $q\equiv3\pmod{4}$.
\end{proof}

We now show that $\PSL{q}$ is not the automorphism group of a chiral polytope of rank $4$ whenever $q=p^d\equiv3\pmod{4}$ and $d>1$ is an (odd) prime power.
\begin{lem}\label{lem:subsg}
If $d$ is an (odd) prime power and $q=p^d\equiv3\pmod{4}$, then $\PSL{q}$ is not the automorphism group of a chiral $4$-polytope.
\end{lem}
\begin{proof}
Suppose $\set{\s{1},\s{2},\s{3}}$ is the set of the distinguished generators of $\Gamma$, the automorphism group of a chiral $4$-polytope with $\Gamma\cong\PSL{q}$.
By Lemma \ref{lem:FixedPointFree}, none of the maximal parabolic subgroups $\gen{\s{1},\s{2}}$ and $\gen{\s{2},\s{3}}$ are subgroups of $\AGL{q}$.
So we may assume that $\gen{\s{1},\s{2}}\leq\PSL{p^{e_1}}$ and $\gen{\s{2},\s{3}}\leq\PSL{p^{e_2}}$, where $e_i<d$.
Indeed, if $\gen{\s{1},\s{2}}$ or $\gen{\s{2},\s{3}}$ is a subgroup of one of $A_5$, $A_4$ or $S_4$, then it is a subgroup of a subfield subgroup because none of $A_5$, $A_4$ nor $S_4$ is maximal in $\PSL{p^d}$ when $d>2$.
Let $F=\GF{p^{\Lcm{e_1,e_2}}}$, that is, $F$ is the smallest field containing all the generators $\s{1}$, $\s{2}$ and $\s{3}$.
Assume, without loss of generality, that $\gen{\s{1},\s{2}}$ stabilises the unique subfield $\GF{p^{e_1}}$ of $F$.
Since $d$ is a prime power, we have $\Lcm{e_1,e_2}=\Max{e_1,e_2}<d$ so $F\not=\GF{q}$.
All involutions of $\PSL{p^{e_1}}$ are conjugate, so $\s{1}$ and $\s{2}$ can be chosen so that
\begin{align*}
\s{1}\s{2}=
\left[
\begin{matrix}
0&1\\-1&0
\end{matrix}
\right].
\end{align*}
Let $\s{2}$ and $\s{3}$ be defined by
\begin{align*}
\s{2}=
\left[
\begin{matrix}
a&b\\c&d
\end{matrix}
\right]
&&\text{and}&&
\s{3}=
\left[
\begin{matrix}
w&x\\y&z
\end{matrix}
\right],
\end{align*}
where $a,b,c,d\in\GF{p^{e_1}}\leq F$, $w,x,y,z\in\GF{q}$ and $\Det{\s{2}}=\Det{\s{3}}=1$.
Let $\omega$ be the trace of $\s{3}$, that is $\omega=w+z$.
All subgroups isomorphic to $\PSL{p^{e_2}}$ are conjugate in $\PSL{q}$, so $\omega\in\GF{p^{e_2}}$ because the trace is invariant under conjugation.
Since $\s{1}\s{2}\s{3}$ is an involution, we have that $x=y$.
Moreover, since $\s{2}\s{3}$ is an involution, we have 
\begin{align*}
aw+x(b+c)+dz=0.
\end{align*}
That is,
\begin{align}\label{eq:d1}
w(d-a)=d\omega+x(b+c)
\end{align}
and
\begin{align}\label{eq:d2}
z(d-a)=-a\omega-x(b+c).
\end{align}
If $d-1=0$, then $x^2=wz-1$ and all the coefficients of $\s{3}$ belong to $F$ or to the quadratic extension of $F$.
Suppose now that $d-a\not=0$.
Since the determinant of $\s{3}$ equals $1$, we have $wz-xy=1$, and multiplying by $(d-a)^2$ gives us
\begin{align}\label{eq:d3}
wz(d-a)^2-x^2(d-a)^2=(d-a)^2.
\end{align}
Putting \eqref{eq:d1}, \eqref{eq:d2} and \eqref{eq:d3} together yields the following quadratic equation in $F$
\begin{align*}
-x^2(b+c)^2+x(b+c)(a+d)+ad\omega-x^2(d-a)^2=(d-a)^2.
\end{align*}
Hence $x$ has either a solution in $F$ or in the quadratic extension of $F$.
It follows that all coefficients of $\s{3}$ are either scalars of $F$ or of the quadratic extension of $F$.
In the latter case, $\s{3}$ would belong to $\PSL{p^{2\Lcm{e_1,e_2}}}$, which is impossible since $d$ is odd and $\PSL{p^e}\leq\PSL{q}$ implies $e$ odd.
If all coefficients of $\s{3}$ belong to $F$, then $\s{3}$ stabilises $F$ and thus $\gen{\s{1},\s{2},\s{3}}\cong\PSL{F}\not\cong\PSL{q}$, a contradiction.
As a conclusion, we showed that there is no chiral $4$-polytope with automorphism group $\PSL{q}$ whenever $q=p^d\equiv3\pmod{4}$ and $d>1$ is a prime power.
\end{proof}

Note that this statement does not hold with the assumption that $q\equiv1\pmod{4}$ even if the parabolic subgroups are fixed point free.
For instance, if $q=169$ and $\gen{\s{1},\s{2}}\cong\PSL{13}$, then by \cite[Proposition 6]{Hypermaps}, $\gen{\s{2},\s{3}}$ is a subgroup of $\PSL{169}$ but not of $\PSL{13}$.
However, $\gen{\s{2},\s{3}}$ can be $A_5$ or $\PGL{13}$.
A {\sc Magma}~\cite{BCP97} computation shows that there exists indeed a chiral polytope of type $[7,3,5]$, whose parabolic subgroups are $\PSL{13}$ and $A_5$.
A complete computation shows that there are exactly $44$ chiral polytopes of rank four with $\PSL{169}$ as automorphism group.
Table \ref{table:169} gives those polytopes, up to duality, with their parabolic subgroups.
\begin{table}
{\renewcommand{\arraystretch}{1.0}%
  \centering
$\begin{array}{|c|c|c|c|}
\hline
\text{type}&\#&\gen{\s{1},\s{2}}&\gen{\s{2},\s{3}}\\
\hline
[4,3,5]&2&S_4&A_5\\

[6,3,5]&2&E_{13}:C_6&A_5\\

[7,3,5]&2&\PSL{13}&A_5\\

[13,3,5]&2&\PSL{13}&A_5\\

[14,3,5]&2&\PGL{13}&A_5\\

[21,42,21]&6&E_{169}:C_{42}&E_{169}:C_{42}\\

[28,28,28]&6&E_{169}:C_{28}&E_{169}:C_{28}\\

[84,84,84]&12&E_{169}:C_{84}&E_{169}:C_{84}\\
\hline
\end{array}$}

\caption{Chiral polytopes of $\PSL{169}$ and their parabolic subgroups}
\label{table:169}
\end{table}

Let $p\equiv3\pmod{4}$.
In the next lemma, we reduce the list of the possible types that chiral polytopes of rank $4$ with $\PSL{p}$ as automorphism group might have.

\begin{lem}\label{lem:reduced_list}
  Let $p\equiv3\pmod{4}$ be a prime power.
  Up to duality, the type of a chiral polytope of rank four with automorphism group isomorphic to $\PSL{p}$ is one of $[3,5,3]$, $[5,3,5]$ or $[5,3,4]$.
\end{lem}
\begin{proof}
By Lemma \ref{lem:FixedPointFree}, the parabolic subgroups of a chiral $4$-polytope $\poly$ with $\PSL{p}$ as automorphism group must be one of $A_4$, $S_4$ or $A_5$.
Hence the Schl\"{a}fli symbols of $\poly$ belong to the set $\set{3,4,5}$.
As a first step, we can eliminate the types corresponding to finite Coxeter groups, namely the $4$-simplex $[3,3,3]$ with automorphism group $A_5$; the tesseract $[3,3,4]$ and its dual, the $16$-cell $[4,3,3]$, with automorphism group of order $192$; the $120$-cell $[3,3,5]$ its dual, the $600$-cell $[5,3,3]$, with automorphism group of order $14400$ and the $24$-cell $[3,4,3]$ with automorphism group of order $1152$.
By a computation using the L2-quotient algorithm \cite{J2015}, the Coxeter even group $[4,3,4]^+$  has no quotient isomorphic to $\PSL{q}$.
Since $A_5$ is not a quotient of $[4,5]^+$, we can eliminate $[3,4,5]$, $[3,5,4]$, $[4,4,5]$, $[4,5,4]$, $[4,5,5]$, $[5,4,3]$, $[5,4,4]$, $[5,4,5]$ and $[5,5,4]$.
Moreover, none of $A_4$, $S_4$ or $A_5$ is a quotient of $[4,4]^+$, so we can also rule out $[3,4,4]$, $[4,4,3]$ and $[4,4,4]$.
It remains to consider $[5,5,5]$ where  $\gen{\s{1},\s{2}}\cong\gen{\s{2},\s{3}}\cong A_5$ and $[3,5,5]$ where $\gen{\s{1},\s{2}}\cong\gen{\s{2},\s{3}}\cong A_5$.
Suppose $\gen{\s{1},\s{2},\s{3}}$ is a quotient of the even Coxeter group $[5,5,5]^+$, where $\gen{\s{1},\s{2}}\cong A_5\cong\gen{\s{2},\s{3}}$.
It is easy to check that $\s{1}\s{2}^2\s{1}^3\s{2}^2\s{1}^2=\epsilon$ and $\s{2}\s{3}^2\s{2}^3\s{3}^2\s{2}^2=\epsilon$.
Adding those relations to the presentation of $[5,5,5]^+$ gives a group of order $7200$ isomorphic to $[3,3,5]^+$.
Now if $\gen{\s{1},\s{2},\s{3}}$ is a quotient of the even Coxeter group $[5,5,3]^+$, where $\gen{\s{1},\s{2}}\cong A_5\cong\gen{\s{2},\s{3}}$, then adding the relation $\s{1}\s{2}^2\s{1}^3\s{2}^2\s{1}^2=\epsilon$ gives the same group of order $7200$ as above.
\end{proof}

%We already dealt with the case where $q$ is even in the previous section so in this section we assume $q$ odd.
%The groups $\PSL{q}$ are much harder to deal with.
%Indeed, there exist some prime powers $q$ for which there is no chiral $4$-polytope with $\PSL{q}$ as automorphism group.
%In fact there are infinitely many such prime powers $q$.
%
%Algebraic conditions on $q$ for which $\PSL{q}$ is the automorphism group of a chiral $4$-polytope are given in the end of this section.
%
%This section opens with another construction of a chiral $4$-polytope with groups $\PSL{q}$, when $q\equiv1\pmod{4}$, following the same method to that of polytopes with groups $\PGL{q}$.
%There are two cases to distinguish, namely where $q\equiv1\pmod{8}$ and $q\equiv5\pmod{8}$, giving different types in both cases.
%Next, we'll see that whenever $q\equiv3\pmod{4}$, chiral $4$-polytopes for $\PSL{q}$ only concern those with type $[3,5,3]$, $[5,3,5]$, $[5,3,4]$ and $[4,3,5]$.

In Lemma \ref{lem:reduced_list}, we proved that if $q=p\equiv3\pmod{4}$, then each chiral $4$-polytope with $\PSL{q}$ as automorphism group must have one of the following types: $[5,3,5]$, $[3,5,3]$ or $[5,3,4]$ (or its dual $[4,3,5]$).
Moreover, by Lemma \ref{lem:subsg}, if $q\equiv3\pmod{4}$ and $\PSL{q}$ is the automorphism group of a chiral $4$-polytope, then $q=p$.
We now classify the chiral $4$-polytopes with $\PSL{q}$ as automorphism group, when $q=p^d\equiv3\pmod{4}$ and $d\geq1$ is a prime power, by counting the number of those chiral $4$-polytopes of type $[5,3,5]$, $[3,5,3]$ and $[5,3,4]$.

%%%%%%%%%%%%%%%%%%%%%%%%%%%%%%%%%%%%%%%%%%%%%%%%%%%%%%%%%%%%%%%%%%%%%%%%%%%%%%%%%%%%%%%%%%%%%%%%
%
\subsubsection{Polytopes of type $[5,3,4]$}\label{sec:534}
%
%%%%%%%%%%%%%%%%%%%%%%%%%%%%%%%%%%%%%%%%%%%%%%%%%%%%%%%%%%%%%%%%%%%%%%%%%%%%%%%%%%%%%%%%%%%%%%%%

%\begin{commentJ}
%Uniformiser : even subgroup or rotations subgroup
%
%Uniformiser : $\Autgp_0$ instead of $\gen{\s{1},\s{2}}$ and $\Autgp_1$ instead of $\gen{\s{2},\s{3}}$ everywhere ?
%\end{commentJ}

The notation in this section is mostly taken from \cite{JL2013} and \cite{JLM2011}.
Let $\Gamma:=[5,3,4]$ be the Coxeter group of type $[5,3,4]$ and let $\Gamma^+$ be the even subgroup of $\Gamma$, that is the group with presentation
\begin{align*}
  \Gamma^+:=\gen{\s{1},\s{2},\s{3}|\s{1}^5,\s{2}^3,\s{3}^4,(\s{1}\s{2})^2,(\s{2}\s{3})^2,(\s{1}\s{2}\s{3})^2}.
\end{align*}
Let $\Gamma_0^+ := \langle \alpha,\beta\rangle$.
Recall that a perfect group is a group that is equal to its own commutator subgroup or equivalently, whose abelian quotient is trivial.
Recall also that $G^{ab}$ is the quotient of a group $G$ by its commutator subgroup.

\begin{lem}\label{lem:abquot}
If $G\rightarrow H$ is a group epimorphism, then $\ab{H}$ is isomorphic to a subgroup of $\ab{G}$.
In particular, if $G$ is perfect, then $H$ is perfect.
\end{lem}
\begin{proof}
The sequence $G\rightarrow H\rightarrow0$ is exact.
Since the abelianization functor is right exact, the sequence $\ab{G}\rightarrow\ab{H}\rightarrow0$ is exact too.
So $\ab{G}\rightarrow\ab{H}$ is an epimorphism.
\end{proof}

Let $\GFb{p}$ be the algebraic closure of $\GF{p}$, that is
\begin{align*}
\GFb{p}=\bigcup_{i=1}^\infty\GF{p^i}.
\end{align*}
Define $\bar{L}_p:=\PSL{\GFb{p}}$.
When the characteristic $p$ does not need to be mentioned, we simply omit it and write $\bar{L}$ instead.

\begin{lem} \label{lem:perfect}
Any homomorphism $\Gamma^+\rightarrow\bar{L}$ preserving the order of $\s{3}$ is an epimorphism onto $\PSL{q}$ or $\PGL{q}$, for some prime power $q=p^d$.
Moreover, if $q>5$, then any epimorphism $\Gamma^+\rightarrow\PSL{q}$ or $\PGL{q}$ preserves the order of $\s{1}$, $\s{2}$ and $\s{3}$.
\end{lem}

\begin{proof}
Let $\theta:\Gamma^+\rightarrow\bar{L}$ be a non trivial homomorphism.
Since $\Gamma^+$ is finitely generated, its image must also be finitely generated, and thus isomorphic to a subgroup of $\PSL{q}$.
Since $A_5$ is perfect, the restriction of $\theta$ to $\Gamma_0^+:=[5,3]^+$ is an isomorphism onto $A_5$, so the orders of $\s{1}$ and $\s{2}$ are preserved.
The image of $\theta$ cannot be isomorphic to $A_5$, $S_4$, $S_3$ or $C_2$.
By Lemma \ref{lem:abquot}, the abelian quotient of the image of $\Gamma^+$ must be isomorphic to a subgroup of $\ab{\Gamma^+}\cong C_2$.
We have $\ab{\PSL{q}}\cong\{1\}$ since $\PSL{q}$ is perfect; for $q$ odd, $\ab{\PGL{q}}\cong\C{2}$; $\ab{A_4}\cong\ab{A_3}\cong\C{3}$; $\ab{\AGL{q}}\cong\C{q-1}$ and $\ab{\C{k}}=\C{k}$.
Clearly, $\theta(\Gamma^+)$ cannot be dihedral and so must be isomorphic to $\PSL{q}$ or $\PGL{q}$.

If a homomorphism $\Gamma^+\rightarrow\bar{L}$ does not preserve the order of $\s{3}$, then its image is a quotient of $A_5\times\C{2}$ which has no subgroups isomorphic to $\PSL{q}$ or $\PGL{q}$ unless $q\leq5$.
\end{proof}

We show that any epimorphism from $\Gamma^+$ onto a projective special linear group gives rise to a group satisfying the intersection condition $\ref{C3'}$.
Note that the regular case was already covered by \cite{intersection}, for the $[5,3,4]$ Coxeter group, but also for the $[5,3,5]$, $[3,5,3]$ and $[4,3,4]$ Coxeter groups.

\begin{lem}\label{lem:intersection_534}
If $q\not=4,5$, then any quotient of $\Gamma^+$ isomorphic to $\PSL{q}$ satisfies the intersection condition.
\end{lem}

\begin{proof}
If $q=4$ or $5$, the third generator of $\Gamma^+$ is mapped to the identity element, so the intersection condition fails.
So we assume that $q\not=4,5$.
The only non trivial case to verify in (C3') is $\gen{\s{1},\s{2}}\cap\gen{\s{2},\s{3}}=\gen{\s{2}}$.
  The image of the subgroup $\gen{\s{1},\s{2}}$ is isomorphic to the icosahedral subgroup and the image of $\gen{\s{2},\s{3}}$ is isomorphic to $S_4$, so the intersection $\langle\s{1},\s{2}\rangle\cap\langle\s{2},\s{3}\rangle$ is one of $C_3$, $D_{6}$ or $A_4$.
  
  Suppose first that the intersection is isomorphic to a dihedral subgroup of order $6$.
  Let $\tau_1:=\s{1}\s{2}^2\s{1}^2\s{2}^2\s{1}^2$ and $\tau_2:=\s{3}\s{2}^2\s{3}^2$.
  It is an easy exercise to verify that $\tau_1$ and $\tau_2$ are involutions that transpose $\s{2}$ and $\s{2}^{-1}$.
  Since there is exactly one dihedral subgroup of order $6$, that contains a cyclic group of order $3$ in the icosahedral group and exactly one dihedral subgroup of order $6$, that contains a cyclic group of order $3$ in $S_4$, those two involutions must belong to the intersection.
  Hence $\tau_2=\tau_1\s{2}^k$ where $k\in\{1,2,3\}$.
  Adding each of those relations to $\Gamma^+$ gives respectively $A_5$ and twice the identity group.  
  Hence the intersection cannot be a dihedral subgroup of order $6$ since $q\not=4,5$.
  
  Suppose now that the intersection is isomorphic to $A_4$.  
  Let $\tau_1:=\s{1}^2\s{2}^{-1}\s{1}$, $\tau_2:=\s{1}\s{2}^{-1}\s{1}^2$ and $\tau_3:=\s{2}^2\s{3}^2$.
  It is an easy exercise to verify that, for $i=1,2,3$, $\tau_i\s{2}$ is an involution and $\tau_i$ is an element of order $3$.
  In $A_5$, there are exactly two distinct subgroups $A_4$ containing a cyclic subgroup of order $3$ whereas in $S_4$ there is exactly one $A_4$ containing a cyclic subgroup of order $3$.
  The three involutions of $A_4$ are all conjugate in $A_4$, hence we either have, for $i=1,2$, $\tau_i\s{2}=\tau_3\s{2}$, $(\tau_i\s{2})^{\s{2}}=\tau_3\s{2}$ or $(\tau_i\s{2})^{\s{2}^2}=\tau_3\s{2}$.
  Equivalently, we either have, for $i=1,2$
  \begin{align*}
  \tau_i=\tau_3, && \tau_i^{\s{2}}=\tau_3 && \text{or} && \tau_i^{\s{2}^2}=\tau_3.
  \end{align*}
  This gives six relations to consider.
  Adding each of them to the presentation of $\Gamma^+$ always gives a cyclic subgroup of order $2$.
  Hence the intersection cannot be $A_4$ or $D_6$ and it must be a cyclic subgroup of order $3$.  

\end{proof}

In order to count the number of pairwise non isomorphic chiral polytopes of type $[5,3,4]$, we consider first the homomorphisms $\Gamma_0^+\rightarrow\bar{L}$.
The following result has already been proved in \cite{JL2013} for example.

\begin{lem}[\cite{JL2013}]\label{lem:one_conjugacy_class_of_A5}
There is only one conjugacy class of icosahedral subgroups in $\bar{L}$.
\end{lem}
%\begin{proof}
%Any two icosahedral subgroups are both contained in a finite subgroup $\PSL{q}$ of $\bar{L}$.
%They are conjugate in $\PGL{q}$ which is contained in $\PSL{q^2}\leq\bar{L}$.
%\end{proof}

Note that any non-trivial homomorphism $\Gamma_0^+\rightarrow\bar{L}$ is an isomorphism with $A_5$.
This is because $A_5$ is the only perfect non-trivial subgroup of itself.
Hence any non-trivial homomorphism $\Gamma^+_0\rightarrow\bar{L}$ is an embedding.
We define two embeddings $\psi_1$, $\psi_2:\Gamma_0^+\rightarrow\bar{L}$ to be {\em equivalent} if and only if there exists $x\in\bar{L}$ such that $\psi_1(\s{1})=\psi_2(\s{1})^x$ and $\psi_1(\s{2})=\psi_2(\s{2})^x$.

\begin{lem}[\cite{JL2013}]
There are exactly two equivalence classes of embeddings $\Gamma_0^+\rightarrow\bar{L}$.
\end{lem}

This equivalence relation allows us the restrict our attention only to the extensions of those two equivalence classes of embeddings $\Gamma_0^+\rightarrow\bar{L}$.
We clarify this statement in the following lemma:

\begin{lem}\label{lem:equiv_quotients}
Let $\psi_1,\psi_2:\Gamma_0^+\rightarrow\bar{L}$ be two equivalent homomorphisms.
If $\theta_2$ is an extension of $\psi_2$, then there exists an extension $\theta_1$ of $\psi_1$ such that the quotients $\bar{L}/\ker{\theta_1}$ and $\bar{L}/\ker{\theta_2}$ are equivalent, in the sense that there exists $x\in\bar{L}$ such that $\theta_1(\s{1})=\theta_2(\s{1})^x$, $\theta_1(\s{2})=\theta_2(\s{2})^x$ and $\theta_1(\s{3})=\theta_2(\s{3})^x$.
\end{lem}

\begin{proof}
Since $\psi_1$ and $\psi_2$ are equivalent, there exists $x\in\bar{L}$ such that $\psi_1(\s{1})=\psi_2(\s{1})^x$ and $\psi_1(\s{2})=\psi_2(\s{2})^x$.
The homomorphism $\theta_1:\Gamma_0^+\rightarrow\bar{L}$ defined by $\theta_1(\s{1})=\psi_1(\s{1})$, $\theta_1(\s{2})=\psi_1(\s{2})$ and $\theta_1(\s{3})=\theta_2(\s{3})^x$ is of course an extension of $\psi_1$ equivalent to $\theta_2$.
\end{proof}
Clearly, two equivalent quotients (satisfying conditions \eqref{C1} to \eqref{C3}) induce two isomorphic chiral polytopes.

Let $F=\mathbb{F}_p$ or $\mathbb{F}_{p^2}$ according as $p\equiv\pm1\pmod{5}$ and $p\equiv\pm1\pmod{8}$ or not.
  That is $F=\mathbb{F}_p$ if $p\equiv\pm1,\pm9\pmod{40}$ and $F=\mathbb{F}_{p^2}$ if $p\equiv\pm3,\pm7,\pm11,\pm13,\pm17,\pm19\pmod{40}$, so that $F$ is the smallest subfield of $\overline{F}_p$ containing a square root of $5$ and a square root of $2$, or equivalently for which $\PSL{F}$ contains elements of order $5$ and $4$.
  Let $t_1=\frac{-1+\sqrt{5}}{2}$ and $t_2=\frac{-1-\sqrt{5}}{2}$.
   For each $i=1,2$, it is well known that there exist two elements $\scal{1}_i,\scal{2}_i\in F$ such that $\scal{1}_i^2+\scal{2}_i^2=t_i^2-3$.
  Note that, for $i=1,2$, we cannot have $\scal{1}_i=\scal{2}_i=0$ since this gives $t_i=-2$ and so $4=t_i^2=1-t_i=3$, which is impossible.
  Without loss of generality, we may assume $\scal{1}_i\not=0$.

\begin{lem}
For $i=1,2$, the two embeddings $\psi_i:\Gamma^+_0\rightarrow\bar{L}$ defined by:
\begin{align*}
  \s{1}\mapsto\frac{1}{2}\left[\begin{matrix}t_i-\scal{2}_i&\scal{1}_i+1\\\scal{1}_i-1&t_i+\scal{2}_i\end{matrix}\right],&&\s{2}\mapsto\frac{1}{2}\left[\begin{matrix}\scal{1}_i+1&t_i+\scal{2}_i\\\scal{2}_i-t_i&1-\scal{1}_i\end{matrix}\right].
  \end{align*}
are not equivalent.
\end{lem}

\begin{proof}
It is easy to see that $\psi_i((\s{1}\s{2})^2)=1_{\bar{L}}$ hence $\psi_i$ is well defined.
Moreover, they are not equivalent since for $i=1,2$, $\psi_i(\s{1})$ has trace $t_i$ and the traces are invariant under conjugation.
\end{proof}

Note that we could have taken $t_1=-\frac{-1+\sqrt{5}}{2}$ and $t_2=-\frac{-1-\sqrt{5}}{2}$, but those give equivalent embeddings as those described above.

\begin{thm} \label{lem:quotients}
The only chiral polytopes of type $[5,3,4]$ with $\PSL{p}$ as automorphism group are given by the following list:
\begin{enumerate}
\item if $p\equiv1,9\pmod{40}$ and $1\pm\sqrt{5}$ are both squares in $\GF{p}$, then there are four chiral polytopes with automorphism group isomorphic to $\PSL{p}$.
\item if $p=31,39\pmod{40}$, then there are two chiral polytopes with automorphism group isomorphic to $\PSL{p}$.
\end{enumerate}
%\begin{enumerate}
%      \item if $p\equiv1,9\pmod{40}$,
%      \begin{enumerate}
%        \item if $1\pm\sqrt{5}$ are both squares in $\GF{p}$, there are four quotients isomorphic to $\PSL{p}$;
%        \item if $1\pm\sqrt{5}$ are both nonsquares in $\GF{p}$, there are two quotients isomorphic to $\PSL{p^2}$,
%      \end{enumerate}
%      \item if $p\equiv31,39\pmod{40}$ there are two quotients isomorphic to $\PSL{p}$ and one with quotient isomorphic to $\PSL{p^2}$;
%       % \item three quotients isomorphic to $\PGL{p}$;
%%      \item if $p\equiv\pm11,\pm19\pmod{40}$ and $1\pm\sqrt{5}$ are both non squares in $\mathbb{F}_{p^2}$ there is one quotient isomorphic to $\PSL{p^4}$;
%      \item if $p\equiv11,19\pmod{40}$ there is one quotient isomorphic to $\PSL{p^2}$ and two quotients isomorphic to $\PGL{p}$;
%      \item if $p\equiv21,29\pmod{40}$
%      \begin{enumerate}
%      	\item if $1\pm\sqrt{5}$ are both squares in $\mathbb{F}_{p}$, there are two quotients isomorphic to $\PSL{p^2}$;
%		\item if $1\pm\sqrt{5}$ are both non squares in $\mathbb{F}_{p}$, there are four quotients isomorphic to $\PGL{p}$;
%	  \end{enumerate}
%      \item if $p\equiv\pm3,\pm7,\pm13,\pm17\pmod{40}$,
%      \begin{enumerate} 
%        \item if $1\pm\sqrt{5}$ are both squares in $\mathbb{F}_{p^2}$, there are two quotients isomorphic to $\PSL{p^2}$;
%        \item if $1\pm\sqrt{5}$ are both nonsquares in $\mathbb{F}_{p^2}$, there is one quotient isomorphic to $\PSL{p^4}$.
%      \end{enumerate}
%\end{enumerate}
\end{thm}

\begin{proof}
  The mapping defined by
  \begin{align*}
  \s{1}\mapsto\psi_i(\s{1})&&\s{2}\mapsto\psi_i(\s{2})&&\s{3}\rightarrow\left[\begin{matrix}w&x\\y&z\end{matrix}\right],     
  \end{align*}
  where $wz-xy=1$ and $w+z=\sqrt{2}$, is a homomorphism if and only if it preserves the relations $\s{3}^5=(\s{2}\s{3})^2=(\s{1}\s{2}\s{3})^2=1$.
  A similar calculation as in~\cite{JL2013} shows that the solutions in $\overline{\mathbb{F}}_p$ of the quadratic equation
 \begin{align}\label{eq:quad}
 (\scal{1}_i^2+\scal{2}_i^2)x^2+\sqrt{2}\scal{2}_ix+\frac{1}{2}(\scal{1}_i^2+1)=0
 \end{align}
 in $x$ are in one-to-one correspondence with the extensions $\Gamma^+\rightarrow\bar{L}$ that preserve the order of $\s{1}$, $\s{2}$ and $\s{3}$.
Recall that $F=\mathbb{F}_p$ if $p\equiv\pm1,\pm9\pmod{40}$ and $F=\mathbb{F}_{p^2}$ if $p\equiv\pm3,\pm7,\pm11,\pm13,\pm17,\pm19\pmod{40}$.
For $i=1,2$, the determinant $\Delta_i$ of \eqref{eq:quad} is given by
\begin{align*}
    \Delta_i&=2\scal{2}_i^2-2(\scal{1}_i^2+\scal{2}_i^2)(\scal{1}_i^2+1)\\
    &=2\scal{2}_i^2-2\scal{1}_i^4-2\scal{1}_i^2-2\scal{1}_i^2\scal{2}_i^2-2\scal{2}_i^2\\
	&=\scal{1}_i^2(-2(\scal{1}_i^2+\scal{2}_i^2)-2)\\
	&=\scal{1}_i^2(2(t_i+2)-2)\\
    &=\scal{1}_i^2(1\pm\sqrt{5}).
  \end{align*}
  Since $\scal{1}_i\not=0$, the determinant $\Delta_i$ is a square in $F$ if and only if $1\pm\sqrt{5}$ is a square in the same field.
  Since $\Delta_i=0$ implies $t_i=-1$ and $1=t_i^2=1-t_i=2$, we may assume $\Delta_i\not=0$.
  We now have $\Delta_1\Delta_2=-4(a_1a_2)^2=-(2a_1a_2)^2$.
  
  Suppose first that $-1$ is a square, equivalently $p\equiv1,9\pmod{40}$ and $F=\GF{p}$ or $p\equiv\pm3,\pm7,\pm9,\pm11,\pm13,\pm17,\pm19\pmod{40}$ and $F=\GF{p^2}$.
  In the latter case, the images of the epimorphisms are $\PGL{p}$ or $\PSL{p^{2n}}$ so we may assume $p\equiv1,9\pmod{40}$.
  If both $\Delta_1$ and $\Delta_2$ are square in $F$, then there are four epimorphisms onto $\PSL{p}$.
  Since only the identity of $\Aut{\PSL{p}}\cong\PGL{p}$ preserves an icosahedral subgroup of $\PSL{p}$ and the trace of the image of $\s{1}$ is preserved, the four epimorphisms give four non equivalent quotients isomorphic to $\PSL{p}$.
  If both $\Delta_1$ and $\Delta_2$ are non squares in $\GF{p}$, then $\Delta_1$ and $\Delta_2$ are squares in $\GF{p^2}$ and we obtain epimorphisms onto $\PGL{p}$ or $\PSL{p^2}$.
  Suppose now that $-1$ is not a square, equivalently $p\equiv31,39\pmod{40}$ and $F=\GF{p}$.
  Then exactly one of $\Delta_1$ and $\Delta_2$ is a square (say $\Delta_1$).
  The two epimorphisms corresponding to $\Delta_1$ have their image in $\PSL{p}$.
  By a previous argument, we obtain two non equivalent quotients isomorphic to $\PSL{p}$.
  The two epimorphisms corresponding to $\Delta_2$ are onto $\PSL{p^2}$ or $\PGL{p}$.
  Hence the result follows from Lemma \ref{lem:equiv_quotients}.
  
\end{proof}

\begin{thm}\label{thm:534}
Let $q\equiv3\pmod{4}$ be a prime power.
Then $\PSL{q}$ is the automorphism group of a chiral $4$-polytope with type $[5,3,4]$ if and only if $q=p\equiv31,39\pmod{40}$.
\end{thm}
\begin{proof}
By Lemma \ref{lem:intersection_534} the quotients obtained in Theorem \ref{lem:quotients} satisfy the intersection condition \eqref{C3}.
By Lemma \ref{lem:reduced_list} and Theorem \ref{lem:quotients}, if $\gen{\s{1},\s{2},\s{3}}\cong\PSL{q}$ is the automorphism group of a chiral $4$-polytope, then $q=p\equiv31,39\pmod{40}$.
%Hence they are either the automorphism group of a chiral polytope, or the rotation subgroup of a directly regular polytope.
Suppose there exists an involutory automorphism inverting $\s{1}$, mapping $\s{2}$ onto $\s{1}^2\s{2}$ and fixing $\s{3}$.
If this automorphism is an outer automorphism, then the full automorphism group is $\PGL{q}$ or $\Sigma(2,q)$.
By Theorem \ref{RegRank4}, $\PGL{q}$ is the automorphism group of a regular polytope if and only if $q=5$ and $\Sigma(2,q)$ is not the automorphism group of a regular polytope since $q=p$.
If the automorphism is an inner automorphism, then the full automorphism group is isomorphic to the direct pro\-duct $\PSL{q}\times C_2$, and $\PSL{q}$ is the automorphism group of a (non directly) regular 4-polytope.
By \cite{LS2005a}, $\PSL{q}$ is the automorphism group of a regular polytope if and only if $q\in\set{11,19}$.
For $q=11$ the Sch\"afli type is [3,5,3] and for $q=19$ it is [5,3,5].
Hence the quotients given by Theorem \ref{lem:quotients} are indeed chiral polytopes.
\end{proof}

This gives point \eqref{3139mod40} of Theorem \ref{maintheo}.

\subsubsection{Polytopes of type $[5,3,5]$}

  The quotients isomorphic to $\PSL{q}$ of the even subgroups $\Gamma^+$ of the Coxeter group $\Gamma=[5,3,5]$ have been classified by~\cite{JL2013}.
  The result is summarised in the following theorem.
  \begin{thm}\label{thm:Jones535}
  The only quotients of the subgroup $\Gamma^+$, isomorphic to $\PSL{q}$ are given in the following list.
  \begin{enumerate}
    \item for $p=2$, there are two quotients isomorphic to $\PSL{4}$ and one isomorphic to $\PSL{16}$;
    \item for $p=5$, there are two quotients isomorphic to $\PSL{5}\cong\PSL{4}$;
    \item for $p=19$, there are three quotients isomorphic to $\PSL{19}$;
    \item for $p\equiv\pm1\pmod{5}$ with $p\equiv1,4,5,6,7,9,11,16,17\pmod{19}$, there are either four quotients isomorphic to $\PSL{p}$ if $(7\pm5\sqrt{5})/2$ are both squares in $\mathbb{F}_p$ or two quotients isomorphic to $\PSL{p^2}$ if $(7\pm5\sqrt{5})/2$ are both non squares in $\mathbb{F}_p$;
    \item for $p\equiv\pm1\pmod{5}$ with $p\equiv2,3,8,10,12,13,14,15,18\pmod{19}$, there are two quotients isomorphic to $\PSL{p}$ and one quotient isomorphic to $\PSL{p^2}$;
    \item for $p\equiv2\pmod{5}$, there are either two quotients isomorphic to $\PSL{p^2}$ if $(7\pm5\sqrt{5})/2$ are both squares in $\GF{p^2}$ or one quotient isomorphic to $\PSL{p^4}$ if $(7\pm5\sqrt{5})/2$ are both non squares in $\GF{p^2}$.
  \end{enumerate}
  \end{thm}
 
  %\paragraph{Intersection property}
  
  \begin{lem}\label{lem:inter535}
  If $q\not=4,5,19$, every quotient of $\Gamma^+$ isomorphic to $\PSL{q}$ satisfies the intersection property.
  \end{lem}
  
  \begin{proof}
  The only case to verify is $\gen{\s{1},\s{2}}\cap\gen{\s{2},\s{3}}=\gen{\s{2}}$.
  The image of the subgroups $\gen{\s{1},\s{2}}$ and $\gen{\s{2},\s{3}}$ are both isomorphic to the icosahedral subgroup, so the intersection $\langle\s{1},\s{2}\rangle\cap\langle\s{2},\s{3}\rangle$ is one of $C_3$, $D_6$ or $A_4$.
  Suppose first that the intersection is isomorphic to the dihedral subgroup of order $6$.
  Let $\tau_1:=\s{1}\s{2}^2\s{1}^2\s{2}^2\s{1}^2$ and $\tau_2:=\s{3}\s{2}^2\s{3}^2\s{2}^2\s{3}^2$.
  It is an easy exercise to verify that $\tau_1$ and $\tau_2$ are involutions and permute $\s{2}$ and $\s{2}^{-1}$.
  Since there is only one dihedral subgroup of order $6$ in the icosahedral group that contains a cyclic group of order $3$, those two involutions must belong to the intersection.
  Hence $\tau_2=\tau_1\s{2}^k$ where $k\in\{1,2,3\}$.
  Adding each of those relations to $\Gamma^+$ gives respectively $\PSL{19}$, $A_5\cong\PSL{4}\cong \PSL{5}$ and $A_5\cong\PSL{4}\cong\PSL{5}$.  

  Suppose now that the intersection is isomorphic to $A_4$.  
  Let $\tau_1:=\s{1}^2\s{2}^{-1}\s{1}$, $\tau_2:=\s{1}\s{2}^{-1}\s{1}^2$, $\tau_3:=\s{3}^2\s{2}^{-1}\s{3}$ and $\tau_4:=\s{3}\s{2}^{-1}\s{3}^2$.
  It is an easy exercise to verify that, for $i=1,2$, $\tau_i\s{2}$ and, for $i=3,4$, $\tau_i\s{2}$ are involutions and $\scal{2}_i$ and $\tau_i$ are elements of order three.
  In $A_5$, there are exactly two distinct subgroups $A_4$ containing a cyclic subgroup of order $3$.
  The three involutions of $A_4$ are all conjugate, hence we either have $\scal{2}_i\s{2}=\tau_j\s{2}$, $(\tau_i\s{2})^{\s{2}}=\tau_j\s{2}$ or $(\tau_i\s{2})^{\s{2}^2}=\tau_j\s{2}$, where $i=1,2$ and $j=3,4$.
  Equivalently, we either have
  \begin{align*}
  \tau_i=\tau_j, && \tau_i^{\s{2}}=\tau_j && \text{or} && \tau_i^{\s{2}^2}=\tau_j,
  \end{align*}
  where $i=1,2$ and $j=3,4$.
  This gives $12$ relations to consider.
  Adding each of them to the presentation of $\Gamma^+$ always gives the identity subgroup.
  Hence the intersection cannot be $A_4$ or $D_6$ and it must be a cyclic subgroup of order $3$.  
  \end{proof}
  
  \begin{thm}    \label{thm:535}
  Let $q\equiv3\pmod{4}$ be a prime power.
  The only chiral $4$-polytopes of type $[5,3,5]$ with automorphism group isomorphic to $\PSL{q}$ are given by the following list:
  \begin{enumerate}
  \item if $q=p=19$, there are two chiral $4$-polytopes with automorphism group isomorphic to $\PSL{19}$;
    \item if $q=p\equiv\pm1\pmod{5}$, $p\equiv1,4,5,6,7,9,11,16,17\pmod{19}$ and $(7\pm5\sqrt{5})/2$ are both squares in $\mathbb{F}_p$, then there are four chiral $4$-polytopes of type $[5,3,5]$ with automorphism group isomorphic to $\PSL{p}$;
    \item if $q=p\equiv\pm1\pmod{5}$, $p\equiv2,3,8,10,12,13,14,15,18\pmod{19}$, then there are two chiral $4$-polytopes of type $[5,3,5]$ with automorphism group isomorphic to $\PSL{p}$.
  \end{enumerate}
  \end{thm}
  
  \begin{proof}
  Similar to that of Theorem \ref{thm:534} except that there are two chiral polytopes and one regular polytope of type $[5,3,5]$ isomorphic to $\PSL{19}$.  Note that the regular $[5,3,5]$ is the one described in \cite{LS2008a}.
  \end{proof}
  
  This gives points \eqref{1119mod20_3} and \eqref{1119mod20_4} of Theorem \ref{maintheo}.
  
\subsubsection{Polytopes of type $[3,5,3]$}\label{sec:353}

  The quotients isomorphic to $\PSL{q}$ of the even subgroups $\Gamma^+$ of the Coxeter group $\Gamma=[3,5,3]$ have been classified by~\cite{JLM2011}.
  The result is summarised in the following theorem.
  \begin{thm}\label{thm:Jones353}
  The only quotients of the subgroup $\Gamma^+$, isomorphic to $\PSL{p}$ are given in the following list.
  \begin{enumerate}
    \item for $p=2$, there is one quotient isomorphic to $\PSL{2^4}$;
    \item for $p=5$, there is one quotient isomorphic to $\PSL{5^2}$;
    \item for $p=11$, there is one quotient isomorphic to $\PSL{11}$ and one quotient isomorphic to $\PSL{11^2}$;
    \item for $p\equiv\pm1\pmod{5}$ with $p\equiv1,3,4,5,9\pmod{11}$, there are either four quotients isomorphic to $\PSL{p}$ if $3\pm2\sqrt{5}$ are both squares in $\mathbb{F}_p$ or two quotients isomorphic to $\PSL{p^2}$ if $3\pm2\sqrt{5}$ are both non squares in $\mathbb{F}_p$;
    \item for $p\equiv\pm1\pmod{5}$ with $p\equiv2,6,7,8,10\pmod{11}$, there are two quotients isomorphic to $\PSL{p}$ and one quotient isomorphic to $\PSL{p^2}$;
    \item for $2>p\equiv2\pmod{5}$, there are either two quotients isomorphic to $\PSL{p^2}$ if $3\pm2\sqrt{5}$ are both squares in $\GF{p^2}$ or one quotient isomorphic to $\PSL{p^4}$ if $3\pm2\sqrt{5}$ are both non squares in $\GF{p^2}$.
  \end{enumerate}
  \end{thm}
 
  %\paragraph{Intersection property}
  
  \begin{lem}\label{lem:inter353}
  If $q\not=9,11$, then every quotient of $\Gamma^+$ isomorphic to $\PSL{q}$ satisfies the intersection property.
  \end{lem}

  \begin{proof}
  The only case to verify is $\gen{\s{1},\s{2}}\cap\gen{\s{2},\s{3}}=\gen{\s{2}}$.
  The image of the subgroups $\gen{\s{1},\s{2}}$ and $\gen{\s{2},\s{3}}$ are both isomorphic to the icosahedral subgroup, so the intersection $\langle\s{1},\s{2}\rangle\cap\langle\s{2},\s{3}\rangle$ is $C_5$ or $D_{10}$.
  Suppose that the intersection is isomorphic to the dihedral subgroup of order $10$.
  
  Let $\tau_1:=\s{1}\s{2}^2\s{1}^2\s{2}^2\s{1}^2$ and $\tau_2:=\s{3}\s{2}^2\s{3}^2\s{2}^2\s{3}^2$.
  It is an easy exercise to ve\-ri\-fy that $\tau_1$ and $\tau_2$ are involutions and permute $\s{2}$ and $\s{2}^{-1}$.
  Since there is only on dihedral subgroup of order $10$ in the icosahedral group that contains a cyclic group of order $5$, those two involutions must belong to the intersection.
  Hence $\tau_2=\tau_1\s{2}^k$ where $k\in\{1,2,3,4,5\}$.
  Adding each of those relations to $\Gamma^+$ gives respectively $\PSL{11}$, the identity group, $\PSL{9}$, $\PSL{9}$ and the identity group.  
  Hence the intersection cannot be a dihedral subgroup of order $10$ and it must be a cyclic subgroup of order $5$.  
  \end{proof}
  
  \begin{thm}   \label{thm:353} 
  Let $q\equiv3\pmod{4}$.
  The only chiral $4$-polytopes of type $[3,5,3]$ with automorphism group isomorphic to $\PSL{p}$ are given by the following list:
  \begin{enumerate}
    \item if $q=p\equiv\pm1\pmod{5}$, $p\equiv1,3,4,5,9\pmod{11}$ and $3\pm2\sqrt{5}$ are both squares in $\mathbb{F}_p$, there are four chiral $4$-polytopes of type $[3,5,3]$ with automorphism group isomorphic to $\PSL{p}$;
    \item if $q=p\equiv\pm1\pmod{5}$ and $p\equiv2,6,7,8,10\pmod{11}$, there are two chiral $4$-polytopes of type $[3,5,3]$ with automorphism group isomorphic to $\PSL{p}$;
  \end{enumerate}
  \end{thm}
  
  \begin{proof}
  One can check Table \ref{table:PSL} to verify that $\PSL{11}$ has no chiral polytopes of rank $4$.
  Hence the results is a consequence of Theorem \ref{RegRank4} and Lemma \ref{lem:inter353}.
  \end{proof}
  
  This latter theorem gives points \eqref{1119mod20_1} and \eqref{1119mod20_2} of Theorem \ref{maintheo} and hence finishes the proof of Theorem~\ref{maintheo}.
  
  \begin{table}[h]
%\resizebox{\textwidth}{!}{ 
{\renewcommand{\arraystretch}{1.0}%
\centering
  $\begin{array}{|c|c|c|c|c|}
  \hline
  q&\#&q(4)&p(20)&\text{case(s)}\\
 % \hline
  
  \hline
  4&0&&&\\
%  \hline
  5&0&1&5&\\
%  \hline
  7&0&3&7&\\
%  \hline
  8&2&&&\eqref{main:PGL}\\
%  \hline
  9&0&1&9&\\
%  \hline
  11&0&3&11&\\
%  \hline
  13&6&1&13&\eqref{PSL_1mod4}\\
%  \hline
  16&2&&&\eqref{main:PGL}\\
%  \hline
  17&10&1&17&\eqref{PSL_1mod4}\\
%  \hline
  19&4&3&19&\eqref{1119mod20_2}\\
%  \hline
  23&0&3&3&\\
%  \hline
  25&2&1&5&\eqref{PSL_1mod4}\\
%  \hline
  27&0&3&7&\\
%  \hline
  29&10&1&9&\eqref{PSL_1mod4}\\
%  \hline
  31&6&3&11&\eqref{1119mod20_4}\eqref{3139mod40}\\
%  \hline
  32&6&&&\eqref{main:PGL}\\
%  \hline
  37&12&1&17&\eqref{PSL_1mod4}\\
%  \hline
  41&38&1&1&\eqref{PSL_1mod4}\\
%  \hline
  43&0&3&3&\\
%  \hline
  47&0&3&7&\\
%  \hline
  49&16&1&9&\eqref{PSL_1mod4}\\
%  \hline
  53&12&1&13&\eqref{PSL_1mod4}\\
%  \hline
  59&6&3&19&\eqref{1119mod20_1}\eqref{1119mod20_4}\\
%  \hline
  61&44&1&1&\eqref{PSL_1mod4}\\
%  \hline
  64&12&&&\eqref{main:PGL}\\
%  \hline
  67&0&3&7&\\
%  \hline
  71&10&3&11&\eqref{1119mod20_1}\eqref{1119mod20_4}\eqref{3139mod40}\\
%  \hline
  73&38&1&13&\eqref{PSL_1mod4}\\
%  \hline
  79&8&3&19&\eqref{1119mod20_2}\eqref{1119mod20_4}\eqref{3139mod40}\\
%  \hline
  81&6&1&1&\eqref{PSL_1mod4}\\
%  \hline
  83&0&3&3&\\
  \hline

  \end{array}$
  }
  {\renewcommand{\arraystretch}{1.0}%
  \centering
  $\begin{array}{|c|c|c|c|c|}
  \hline
  q&\#&q(4)&p(20)&\text{case(s)}\\
  \hline
    89&46&1&9&\eqref{PSL_1mod4}\\
%  \hline
  97&56&1&17&\eqref{PSL_1mod4}\\
%  \hline
%  \hline
  101&42&1&1&\eqref{PSL_1mod4}\\
%  \hline
  109&42&1&9&\eqref{PSL_1mod4}\\
%  \hline
  113&52&1&13&\eqref{PSL_1mod4}\\
%  \hline
  121&16&1&1&\eqref{PSL_1mod4}\\
%  \hline
  125&10&1&5&\eqref{PSL_1mod4}\\
%  \hline
  128&18&&&\eqref{main:PGL}\\
%  \hline
  131&6&3&11&\eqref{1119mod20_3}\\
%  \hline
  137&54&1&17&\eqref{PSL_1mod4}\\
%  \hline
  139&2&3&19&\eqref{1119mod20_2}\\
%  \hline
  149&38&1&9&\eqref{PSL_1mod4}\\
%  \hline
  151&8&3&11&\eqref{1119mod20_2}\eqref{1119mod20_4}\eqref{3139mod40}\\
%  \hline
  157&42&1&17&\eqref{PSL_1mod4}\\
%  \hline
  169&44&1&9&\eqref{PSL_1mod4}\\
%  \hline
  173&42&1&13&\eqref{PSL_1mod4}\\
%  \hline
  179&2&3&19&\eqref{1119mod20_4}\\
%  \hline
  181&?&1&1&\eqref{PSL_1mod4}\\
%  \hline
  191&?&3&11&\eqref{1119mod20_3}\eqref{3139mod40}\\
%  \hline
  193&?&1&13&\eqref{PSL_1mod4}\\
%  \hline
  197&?&1&17&\eqref{PSL_1mod4}\\
%  \hline
  199&?&3&19&\eqref{1119mod20_1}\eqref{1119mod20_3}\eqref{3139mod40}\\
%  \hline
  211&?&3&11&\eqref{1119mod20_2}\eqref{1119mod20_4}\\
%  \hline
  223&?&3&3&\\
%  \hline
  227&?&3&7&\\
%  \hline
  229&?&1&9&\eqref{PSL_1mod4}\\
%  \hline
  233&?&1&13&\eqref{PSL_1mod4}\\
%  \hline
  239&?&3&19&\eqref{1119mod20_2}\eqref{1119mod20_3}\eqref{3139mod40}\\
%  \hline
  241&?&1&1&\eqref{PSL_1mod4}\\
%  \hline
  243&?&3&3&\\
%  \hline  
  251&?&3&11&\eqref{1119mod20_1}\eqref{1119mod20_3}\\
  \hline
  \end{array}$
  %}
  }
  \caption{Number of chiral $4$-polytopes with $\PSL{q}$ groups.  When the type is one of $[5,3,5]$, $[3,5,3]$, $[5,3,4]$ or $[4,3,5]$, the last column indicates all the cases to refer to (see Theorem \ref{maintheo})}
  \label{table:PSL}
  \end{table}
  
\section{Conclusion and open problems}\label{sec:conj}
  
The problem of determining the existence of a chiral polytope for the projective linear groups is solved for ``most" of the groups.
Indeed, the smallest groups for which we do not have an answer are $\PSL{3^{15}}$, $\PSL{3^{21}}$, $\PSL{7^{15}}$, $\PSL{11^{15}}$,..., that is, $\PSL{q}$ where $q=p^d\equiv3\pmod{4}$ with $1<d\equiv1\pmod{2}$ not a prime power.
In fact, we believe that the following conjecture provides generators for the automorphism group of a chiral $4$-polytope.
Thanks to {\sc Magma} computations, we have been able to show its validity for many values, up to $q=11^{21}$.

%\begin{commentJ}
%TODO: create the program properly and try for more values ?
%\end{commentJ}

  \begin{conj}\label{conj:1}
Let $p$ be an odd prime and $e_1>1$, $e_2>1$ two odd integers.
There exist two primitive elements $j_1\in\GF{p^{e_1}}$ and $j_2\in\GF{p^{e_2}}$ such that $\Omega:=\omega_1^2\omega_2^2-4(\omega_1^2+\omega_2^2)$ is a square in $\GF{p^{\Lcm{e_1,e_2}}}$, where $\omega_i:=j_i+j_i^{-1}$
\end{conj}

It appears, from {\sc Magma} computations, that the probability that a couple of primitive elements $(j_1,j_2)$ are such that the corresponding $\Omega$ is a square is nearly $1/2$.
This gives a strong argument in favour of the truth of the conjecture.

We now show how we would construct a candidate for the automorphism group of a chiral polytope with $\PSL{q}$ as automorphism group, when $q=p^d$ and $1<d\equiv1\pmod{2}$ is not a prime power.

Let $d$ be an odd integer but not a prime power, and $q=p^d$ where $p$ is an odd prime.
Let $e_1$ and $e_2$ be such that $e_1e_2=d$ and $\Gcd{e_1,e_2}=1$.
Let $j_i$ and $w_i$ be as in Conjecture \ref{conj:1} and let
\begin{align*}
\s{1}:=\frac{1}{2}\left[
\begin{matrix}
\omega_1&\omega_1+2\\\omega_1-2&\omega_1
\end{matrix}
\right],&&
\s{2}:=\frac{1}{2}\left[
\begin{matrix}
2+\omega_1&\omega_1\\-\omega_1&2-\omega_1
\end{matrix}
\right]&&\text{and}&&
\s{3}=\left[\begin{matrix}
w&x\\y&z
\end{matrix}\right],
\end{align*}
where $\Det{\s{3}}=1$ and $w+z=\omega_2$.
First note that $\s{1}\s{2}$ is an involution.
Indeed,
\begin{align*}
\s{1}\s{2}=\left[
\begin{matrix}
0&1\\-1&0
\end{matrix}
\right].
\end{align*}
We are interested in finding solutions for $\s{3}$ such that $\s{1}\s{2}$ and $\s{1}\s{2}\s{3}$ are involutions.
If $\s{1}\s{2}\s{3}$ is an involution, then we must have $x=y$.
Moreover, if $\s{2}\s{3}$ is an involution, then
\begin{align*}
w(\omega_1+2)+z(2-\omega_1)=0.
\end{align*}
That is
\begin{align*}
w=\frac{\omega_2}{2\omega_1}(\omega_1-2)&&\text{and}&&z=\frac{\omega_2}{2\omega_1}(\omega_1+2).
\end{align*}
Since $wz-x^2=1$, we have
\begin{align*}
x^2\omega_1^2+\omega_1^2+\frac{1}{4}\omega_2^2(4-\omega_1^2)=0.
\end{align*}
The determinant of this quadratic extension is $D=\omega_1^2(\omega_1^2\omega_2^2-4(\omega_1^2+\omega_2^2))$.
If Conjecture $\ref{conj:1}$ is true, then $D$ is a square in $\GF{p^{\Lcm{e_1,e_2}}}$ so
\begin{align*}
x=y=\pm\frac{\sqrt{\omega_1^2\omega_2^2-4(\omega_1^2+\omega_2^2)}}{2\omega_1}.
\end{align*}
So one can write a representative solution for $\s{3}$ as
\begin{align*}
\s{3}=\left[\begin{matrix}
\omega_2(\omega_1-2)&\pm\sqrt{\omega_1^2\omega_2^2-4(\omega_1^2+\omega_2^2)}\\
\pm\sqrt{\omega_1^2\omega_2^2-4(\omega_1^2+\omega_2^2)}&\omega_2(\omega_1+2)
\end{matrix}\right].
\end{align*}
It is then easy to see that $\gen{\s{1},\s{2},\s{3}}\cong\PSL{q}$ and $\gen{\s{1}}\cap\gen{\s{2}}\cong\gen{\s{2}}\cap\gen{\s{3}}\cong\set{\epsilon}$.
Moreover, by Theorem \ref{RegRank4}, $\gen{\s{1},\s{2},\s{3}}$ is not the rotation subgroup of a directly regular polytope.
However, it remains to prove that $\gen{\s{1},\s{2}}\cap\gen{\s{2},\s{3}}=\gen{\s{2}}$.
If this intersection contains a dihedral subgroup, then it is a subfield subgroup isomorphic to $\PSL{p}$.
Hence the possibilities for this intersection are $\PSL{p}$, $C_p$ and metacyclic subgroups.
We finish this paper by stating the following less obvious conjecture, for which we have not been able to do any computing experiments.
\begin{conj}
In the previous construction, $\set{\s{1},\s{2},\s{3}}$ satisfies the intersection property, that is, $\gen{\s{1},\s{2}}\cap\gen{\s{2},\s{3}}=\gen{\s{2}}$.
\end{conj}

\section{Acknowledgements}
This work has been financially supported by Marsden Grant UOA1218 from the Royal Society of New Zealand.
The second author would like to thank Marston Conder for his help with the proof of the intersection property in section \ref{sec:353}.

\bibliographystyle{plain}

\end{document}